\documentclass[11pt]{amsart}
\usepackage{eurosym}
\usepackage{amsfonts}
\usepackage{esint}

\newtheorem{theorem}{Theorem}
\theoremstyle{plain}

\newtheorem{corollary}{Corollary}

\newtheorem{definition}{Definition}

\newtheorem{lemma}{Lemma}

\newtheorem{remark}{Remark}

\numberwithin{equation}{section}
\input{tcilatex}

\begin{document}
\title[Integro-differential problems]{On the Cauchy problem for stochastic
integrodifferential parabolic equations in the scale of $L^{p}$-spaces of
generalized smoothness}
\author{R. Mikulevi\v{c}ius and C. Phonsom}
\address{University of Southern California, Los Angeles}
\date{October 23, 2017}

\begin{abstract}
Stochastic parabolic integro-differential problem is considered in the whole
space. By verifying stochastic H\"{o}rmander condition, the existence and
uniqueness is proved in $L_{p}$-spaces of functions whose regularity is
defined by a scalable Levy measure. Some rough probability density function
estimates of the associated Levy process are used as well.
\end{abstract}

\maketitle
\tableofcontents

\section{\protect\bigskip Introduction}

Let $\sigma \in \left( 0,2\right) $ and $\mathfrak{A}^{\sigma }$ be the
class of all nonnegative measures $\pi $\ on $\mathbf{R}_{0}^{d}=\mathbf{R}%
^{d}\backslash \left\{ 0\right\} $ such that $\int \left\vert y\right\vert
^{2}\wedge 1d\pi <\infty $ and 
\begin{equation*}
\sigma =\inf \left\{ \alpha <2:\int_{\left\vert y\right\vert \leq
1}\left\vert y\right\vert ^{\alpha }d{\pi }<\infty \right\} .
\end{equation*}%
In addition, we assume that for $\pi \in \mathfrak{A}^{\sigma },$ 
\begin{eqnarray*}
\int_{\left\vert y\right\vert >1}\left\vert y\right\vert d\pi &<&\infty 
\text{ if }\sigma \in \left( 1,2\right) , \\
\int_{R<\left\vert y\right\vert \leq R^{\prime }}yd\pi &=&0\text{ if }\sigma
=1\text{ for all }0<R<R^{\prime }<\infty .\text{ }
\end{eqnarray*}

Let $\left( \Omega ,\mathcal{F},\mathbf{P}\right) $ be a complete
probability space with a filtration of $\sigma -$algebras on $\mathbb{F}%
=\left( \mathcal{F}_{t},t\geq 0\right) $ satisfying the usual conditions.
Let $\mathcal{R}\left( \mathbb{F}\right) $ be the progressive $\sigma -$%
algebra on $\left[ 0,\infty \right) \times \Omega .$ Let $\left( U,\mathcal{U%
},\Pi \right) $ be a measurable space with $\sigma -$finite measure $\Pi ,%
\mathbf{R}_{0}^{d}=\mathbf{R}^{d}\backslash \left\{ 0\right\} .$ Let $%
p\left( dt,dz\right) $ be $\mathbb{F}-$adapted point measures on $\left( %
\left[ 0,\infty \right) \times U,\mathcal{B}\left( \left[ 0,\infty \right)
\right) \otimes \mathcal{U}\right) $ with compensator $\Pi \left( d\nu
\right) dt.$ We denote the martingale measure $q\left( dt,dz\right) =p\left(
dt,dz\right) -\Pi \left( dz\right) dt.$

In this paper we consider the parabolic Cauchy problem 
\begin{eqnarray}
du\left( t,x\right) &=&\left[ Lu(t,x)-\lambda u\left( t,x\right) +f(t,x)%
\right] dt  \label{mainEq} \\
&&+\int_{U}\Phi \left( t,x,z\right) q\left( dt,dz\right) ,  \notag \\
u\left( 0,x\right) &=&g\left( x\right) ,t\geq 0,x\in \mathbf{R}^{d},  \notag
\end{eqnarray}%
with $\lambda \geq 0$ and integro-differential operator 
\begin{equation*}
L\varphi \left( x\right) =L^{\pi }\varphi \left( x\right) =\int \left[
\varphi (x+y)-\varphi \left( x\right) -\chi _{\sigma }\left( y\right) y\cdot
\nabla \varphi \left( x\right) \right] \pi \left( dy\right) ,\varphi \in
C_{0}^{\infty }\left( \mathbf{R}^{d}\right) ,
\end{equation*}%
where $\pi \in \mathfrak{A}^{\sigma },$ $\chi _{\sigma }\left( y\right) =0$
if $\sigma \in \left[ 0,1\right) ,\chi _{\sigma }\left( y\right) =1_{\left\{
\left\vert y\right\vert \leq 1\right\} }\left( y\right) $ if $\sigma =1$ and 
$\chi _{\sigma }\left( y\right) =1$ if $\sigma \in (1,2).$ The symbol of $L$
is 
\begin{equation*}
{\psi }\left( \xi \right) =\psi ^{\pi }\left( \xi \right) =\int \left[
e^{i2\pi \xi \cdot y}-1-i2\pi \chi _{\sigma }\left( y\right) \xi \cdot y%
\right] \pi \left( dy\right) ,\xi \in \mathbf{R}^{d}.
\end{equation*}%
Note that $\pi \left( dy\right) =dy/\left\vert y\right\vert ^{d+\sigma }\in $%
{$\mathfrak{A}$}$^{\sigma }$ and, in this case, $L=L^{\pi }=c\left( \sigma
,d\right) \left( -\Delta \right) ^{\sigma /2}$, where $\left( -\Delta
\right) ^{\sigma /2}$\ is a fractional Laplacian. The equation (\ref{mainEq}%
) is forward Kolmogorov equation for the Levy process associated to $\psi
^{\pi }$. We assume that $g,f$ and $\Phi $ are resp. $\mathcal{F}_{0}\otimes 
\mathcal{B}\left( \mathbf{R}^{d}\right) $- ,$\mathcal{R}\left( \mathbb{F}%
\right) \otimes \mathcal{B}\left( \mathbf{R}^{d}\right) $- , $\Phi $ is $%
\mathcal{R}\left( \mathbb{F}\right) \otimes \mathcal{B}\left( \mathbf{R}%
^{d}\right) \otimes \mathcal{U}$-measurable.

Let $\mu \in ${$\mathfrak{A}$}$^{\sigma }$ and 
\begin{equation}
c_{1}\left\vert \psi ^{\mu }\left( \xi \right) \right\vert \leq |\psi ^{\pi
}\left( \xi \right) |\leq c_{2}\left\vert \psi ^{\mu }\left( \xi \right)
\right\vert ,\xi \in \mathbf{R}^{d},  \label{3'}
\end{equation}%
for some $0<c_{1}\leq c_{2}$. Given $\mu \in ${$\mathfrak{A}$}$^{\sigma
},p\in \left[ 1,\infty \right) ,s\in \mathbf{R}$, we denote $H_{p}^{s}\left(
E\right) =H_{p}^{\mu ;s}\left( E\right) $ the closure in $L_{p}\left(
E\right) $ of $C_{0}^{\infty }\left( E\right) $ with respect to the norm 
\begin{equation*}
\left\vert f\right\vert _{H_{p}^{\mu ;s}\left( E\right) }=\left\vert 
\mathcal{F}^{-1}\left( 1-\text{Re}\psi ^{\mu }\right) ^{s}\mathcal{F}%
f\right\vert _{L_{p}\left( \mathbf{R}^{d}\right) },\,
\end{equation*}%
where $\mathcal{F}$ is the Fourier transform in space variable. In this
paper, under certain \textquotedblright scalability\textquotedblright\ and
nondegeneracy assumptions (see assumptions \textbf{D}$\left( \kappa
,l\right) ,$ \textbf{B}$\left( \kappa ,l\right) $ below), we prove the
existence and uniqueness of solutions to (\ref{mainEq}) in the scale of
spaces $\mathbb{H}_{p}^{\mu ;s}\left( \mathbf{R}^{d}\right) $). Moreover, \ 
\begin{equation}
\left\vert u\right\vert _{\mathbb{H}_{p}^{s}\left( E\right) }\leq C\left[
\left\vert f\right\vert _{\mathbb{H}_{p}^{s-1}\left( E\right) }+\left\vert
g\right\vert _{\mathbb{H}_{p}^{s-\frac{1}{p}}\left( \mathbf{R}^{d}\right)
}+\left\vert \Phi \right\vert _{\mathbb{H}_{2,p}^{s-\frac{1}{2}}\left(
E\right) }+\left\vert \Phi \right\vert _{\mathbb{B}_{p,pp}^{s-\frac{1}{p}%
}\left( E\right) }\right]  \label{4}
\end{equation}%
if $p\geq 2,$where $\mathbb{B}_{p,pp}^{s}$ is the Besov "counterpart" of $%
\mathbb{H}_{p}^{s}$. This paper is a continuation of \cite{MPh} and \cite%
{Mph2}, where (\ref{mainEq}) with $\Phi =0$ was considered. Since the symbol 
$\psi ^{\pi }\left( \xi \right) $ is not smooth in $\xi $, the standard
Fourier multiplier results do not apply in this case. In order to prove the
estimate involving $\Phi $ in (\ref{4}), we follow the idea of \cite{KK2},
by applying a version of Calderon-Zygmund theorem by associating to $L^{\pi
} $ a family of balls and verifying for it the stochastic H\"{o}rmander
condition (see Theorem \ref{stochHormander} in Appendix). As an example, we
consider $\pi \in ${$\mathfrak{A}$}$^{\sigma }$ defined in radial and
angular coordinates $r=\left\vert y\right\vert ,w=y/r,$ as 
\begin{equation}
\pi \left( \Gamma \right) =\int_{0}^{\infty }\int_{\left\vert w\right\vert
=1}\chi _{\Gamma }\left( rw\right) a\left( r,w\right) j\left( r\right)
r^{d-1}S\left( dw\right) dr,\Gamma \in \mathcal{B}\left( \mathbf{R}%
_{0}^{d}\right) ,  \label{5}
\end{equation}%
where $S\left( dw\right) $ is a finite measure on the unit sphere on $%
\mathbf{R}^{d}$. In \cite{zh}, (\ref{mainEq}) with $g=0$ and $\Phi =0$, was
considered, with $\pi $ in the form (\ref{5}) with $a=1,j\left( r\right)
=r^{-d-\sigma },$ and such that 
\begin{eqnarray*}
&&\int_{0}^{\infty }\int_{\left\vert w\right\vert =1}\chi _{\Gamma }\left(
rw\right) r^{-1-\sigma }\rho _{0}\left( w\right) S\left( dw\right) dr \\
&\leq &\pi \left( \Gamma \right) =\int_{0}^{\infty }\int_{\left\vert
w\right\vert =1}\chi _{\Gamma }\left( rw\right) r^{-1-\sigma }a\left(
r,w\right) S\left( dw\right) dr \\
&\leq &\int_{0}^{\infty }\int_{\left\vert w\right\vert =1}\chi _{\Gamma
}\left( rw\right) r^{-1-\sigma }S\left( dw\right) dr,\Gamma \in \mathcal{B}%
\left( \mathbf{R}_{0}^{d}\right) ,
\end{eqnarray*}%
and (\ref{3'}) holds with $\psi ^{\mu }\left( \xi \right) =\left\vert \xi
\right\vert ^{\sigma },\xi \in \mathbf{R}^{d}$. In this case, $H_{p}^{\mu
;1}\left( E\right) =H_{p}^{\sigma }\left( E\right) $ is the standard
fractional Sobolev space. The solution estimate (\ref{4}) for (\ref{mainEq})
was derived in \cite{zh}, using $L^{\infty }$-$BMO$ type estimate. In \cite%
{KK1}, an elliptic problem in the whole space with $L^{\pi }$ was studied
for $\pi $ in the form (\ref{5}) with $S\left( dw\right) =dw$ being a
Lebesgue measure on the unit sphere in $\mathbf{R}^{d}$, with $0<c_{1}\leq
a\leq c_{2}$, and a set of technical assumptions on $j\left( r\right) $. A
sharp function estimate based on the solution H\"{o}lder norm estimate
(following the idea in \cite{KimDong}) was used in \cite{KK1}.

The paper is organized as follows. In Section 2, the main theorem is stated,
and some examples of the form (\ref{5}) are considered. In Section 3,
auxiliary results on approximation of input functions and some probability
density estimates are presented. In section 4, the main result is proved. In
Appendix, stochastic integrals driven by jump measures are constructed and H%
\"{o}rmander condition discussed.

\section{Notation, Function spaces main results and examples}

\subsection{Notation}

The following notation will be used in the paper.

Let $\mathbf{N}=\{1,2,\ldots\},\mathbf{N}_{0}=\left\{ 0,1,\ldots\right\} ,%
\mathbf{R}_{0}^{d}=\mathbf{R}^{d}\backslash\{0\}.$ If $x,y\in\mathbf{R}^{d},$%
\ we write 
\begin{equation*}
x\cdot y=\sum_{i=1}^{d}x_{i}y_{i},\,|x|=\sqrt{x\cdot x}.
\end{equation*}

We denote by $C_{0}^{\infty}(\mathbf{R}^{d})$ the set of all infinitely
differentiable functions on $\mathbf{R}^{d}$ with compact support.

We denote the partial derivatives in $x$ of a function $u(t,x)$ on $\mathbf{R%
}^{d+1}$ by $\partial_{i}u=\partial u/\partial x_{i}$, $\partial_{ij}^{2}u=%
\partial^{2}u/\partial x_{i}\partial x_{j}$, etc.; $Du=\nabla
u=(\partial_{1}u,\ldots,\partial_{d}u)$ denotes the gradient of $u$ with
respect to $x$; for a multiindex $\gamma\in\mathbf{N}_{0}^{d}$ we denote 
\begin{equation*}
D_{x}^{{\scriptsize \gamma}}u(t,x)=\frac{\partial^{|{\scriptsize \gamma|}%
}u(t,x)}{\partial x_{1}^{{\scriptsize \gamma_{1}}}\ldots\partial x_{d}^{%
{\scriptsize \gamma_{d}}}}.
\end{equation*}
For $\alpha\in(0,2]$ and a function $u(t,x)$ on $\mathbf{R}^{d+1}$, we write 
\begin{equation*}
\partial^{{\scriptsize \alpha}}u(t,x)=-\mathcal{F}^{-1}[|\xi|^{{\scriptsize %
\alpha}}\mathcal{F}u(t,\xi)](x),
\end{equation*}
where 
\begin{equation*}
\mathcal{F}h(t,\xi)=\hat{h}\left(\xi\right)=\int_{\mathbf{R}^{d}}\,\mathrm{e}%
^{-i2\pi\xi\cdot x}h(t,x)dx,\mathcal{F}^{-1}h(t,\xi)=\int_{\mathbf{R}^{d}}\,%
\mathrm{e}^{i2\pi\xi\cdot x}h(t,\xi)d\xi.
\end{equation*}

For $\mu \in ${$\mathfrak{A}$}$^{\sigma }$, we denote $Z_{t}^{\mu },t\geq 0,$
the Levy process associated to $L^{\mu }$, i.e., $Z^{\mu }$ is cadlag with
independent increments and its characteristic function 
\begin{equation*}
\mathbf{E}e^{i2\pi \xi \cdot Z_{t}^{\mu }}=\exp \left\{ \psi ^{\mu }\left(
\xi \right) t\right\} ,\xi \in \mathbf{R}^{d},t\geq 0.
\end{equation*}%
The letters $C=C(\cdot ,\ldots ,\cdot )$ and $c=c(\cdot ,\ldots ,\cdot )$
denote constants depending only on quantities appearing in parentheses. In a
given context the same letter will (generally) be used to denote different
constants depending on the same set of arguments.

\subsection{Function Spaces}

Let $S\left( \mathbf{R}^{d}\right) $ be the Schwartz space of real-valued
rapidly decreasing functions. Let $V$ be a Banach space with norm $%
\left\vert \cdot \right\vert _{V}$. The space of $V-$valued tempered
distribution we denote by $S^{\prime }\left( \mathbf{R}^{d},V\right) $($f\in
S^{\prime }\left( \mathbf{R}^{d},V\right) $ is a continuous $V-$valued
linear functional on $S\left( \mathbf{R}^{d}\right) $). If $V=\mathbf{R}$,
we write $S^{\prime }\left( \mathbf{R}^{d},V\right) =S^{\prime }\left( 
\mathbf{R}^{d}\right) $ and denote by $\left\langle \cdot ,\cdot
\right\rangle $ the duality between $S\left( \mathbf{R}^{d}\right) $ and $%
S^{\prime }\left( \mathbf{R}^{d}\right) $.

For a $V-$valued measurable function $h$ on $\mathbf{R}^{d}$and $p\geq1$ we
denote 
\begin{equation*}
\left|h\right|_{V,p}^{p}=\int_{\mathbf{R}^{d}}\left|h\left(x\right)%
\right|_{V}^{p}dx.
\end{equation*}

We fix $\mu \in ${$\mathfrak{A}$}$^{\sigma }$. Obviously, $\text{Re}\psi
^{\mu }=\psi ^{\mu _{sym}}$, where 
\begin{equation*}
\mu _{sym}\left( dy\right) =\frac{1}{2}\left[ \mu \left( dy\right) +\mu
\left( -dy\right) \right] .
\end{equation*}%
Let 
\begin{equation*}
Jv=J_{\mu }v=(I-L^{\mu _{sym}})v=v-L^{\mu _{sym}}v,v\in \mathcal{S}\left( 
\mathbf{R}^{d},V\right) .
\end{equation*}%
For $s\in \mathbf{R}$ set 
\begin{equation*}
J^{s}v=\left( I-L^{\mu _{sym}}\right) ^{s}v=\mathcal{F}^{-1}[(1-\psi ^{\mu
_{sym}})^{s}\hat{v}],v\in \mathcal{S}\left( \mathbf{R}^{d},V\right) .
\end{equation*}

\begin{equation*}
L^{\mu ;s}v=\mathcal{F}^{-1}\left( -\left( -\psi ^{\mu _{sym}}\right) ^{s}%
\hat{v}\right) ,v\in \mathcal{S}\left( \mathbf{R}^{d},V\right) .
\end{equation*}

Note that $L^{\mu;1}v=L^{\mu_{sym}}v,v\in\mathcal{S}\left(\mathbf{R}%
^{d}\right).$

For $p\in \left[ 1,\infty \right) ,s\in \mathbf{R,}$ we define, following 
\cite{fjs}, the Bessel potential space $H_{p}^{s}\left( \mathbf{R}%
^{d},V\right) =H_{p}^{\mu ;s}\left( \mathbf{R}^{d},V\right) $ as the closure
of $\mathcal{S}\left( \mathbf{R}^{d},V\right) $ in the norm 
\begin{eqnarray*}
\left\vert v\right\vert _{H_{p}^{s}\left( \mathbf{R}^{d},V\right) }
&=&\left\vert J^{s}v\right\vert _{L_{p}\left( \mathbf{R}^{d},V\right)
}=\left\vert \mathcal{F}^{-1}[(1-\psi ^{\mu _{sym}})^{s}\hat{v}]\right\vert
_{L_{p}\left( \mathbf{R}^{d},V\right) } \\
&=&\left\vert \left( I-L^{\mu _{sym}}\right) ^{s}v\right\vert _{L_{p}\left( 
\mathbf{R}^{d},V\right) },v\in \mathcal{S}\left( \mathbf{R}^{d}\right) .
\end{eqnarray*}%
According to Theorem 2.3.1 in \cite{fjs}, $H_{p}^{t}\left( \mathbf{R}%
^{d}\right) \subseteq H_{p}^{s}\left( \mathbf{R}^{d}\right) \,\ $is
continuously embedded if $p\in \left( 1,\infty \right) ,s<t$, $%
H_{p}^{0}\left( \mathbf{R}^{d}\right) =L_{p}\left( \mathbf{R}^{d}\right) $.
For $s\geq 0,p\in \left[ 1,\infty \right) ,$ the norm $\left\vert
v\right\vert _{H_{p}^{s}}$ is equivalent to (see Theorem 2.2.7 in \cite{fjs}%
) 
\begin{equation*}
\left\vert v\right\vert _{H_{p}^{s}}=\left\vert v\right\vert
_{L_{p}}+\left\vert \mathcal{F}^{-1}\left[ (-\psi ^{\mu _{sym}})^{s}\mathcal{%
F}v\right] \right\vert _{L_{p}}.
\end{equation*}

Further, for a characterization of our function spaces we will use the
following construction (see \cite{lofstrom}). We fix a continuous function $%
\kappa :(0,\infty )\rightarrow (0,\infty )$ such that $\lim_{R\rightarrow
0}\kappa \left( R\right) =0,\lim_{R\rightarrow \infty }\kappa \left(
R\right) =\infty $. Assume there is a nondecreasing continuous function $%
l\left( \varepsilon \right) ,\varepsilon >0,$ such that $\lim_{\varepsilon
\rightarrow 0}l\left( \varepsilon \right) =0$ and 
\begin{equation*}
\kappa \left( \varepsilon r\right) \leq l\left( \varepsilon \right) \kappa
(r),r>0,\varepsilon >0.
\end{equation*}%
We say $\kappa $ is a \emph{scaling function }and call $l\left( \varepsilon
\right) ,\varepsilon >0,$ a \emph{scaling factor} of $\kappa $. Fix an
integer $N$ so that $l\left( N^{-1}\right) <1.$

\begin{remark}
\label{re1}For an integer $N>1$ there exists a function $\phi =\phi ^{N}\in
C_{0}^{\infty }(\mathbf{R}^{d})$ (see Lemma 6.1.7 in \cite{lofstrom}), such
that $\mathrm{supp}~\phi =\left\{ \xi :\frac{1}{N}\leq \left\vert \xi
\right\vert \leq N\right\} $ , $\phi (\xi )>0$ if $N^{-1}<|\xi |<N$ and 
\begin{equation*}
\sum_{j=-\infty }^{\infty }\phi (N^{-j}\xi )=1\quad \text{if }\xi \neq 0.
\end{equation*}%
Let 
\begin{equation}
\tilde{\phi}\left( \xi \right) =\phi \left( N\xi \right) +\phi \left( \xi
\right) +\phi \left( N^{-1}\xi \right) ,\xi \in \mathbf{R}^{d}.  \label{pp1}
\end{equation}%
Note that supp$~\tilde{\phi}\subseteq \left\{ N^{-2}\leq \left\vert \xi
\right\vert \leq N^{2}\right\} $ and $\tilde{\phi}\phi =\phi $. Let $\varphi
_{k}=\varphi _{k}^{N}=\mathcal{F}^{-1}\phi \left( N^{-k}\cdot \right) ,k\geq
1,$ and $\varphi _{0}=\varphi _{0}^{N}\in \mathcal{S}\left( \mathbf{R}%
^{d}\right) $ is defined as 
\begin{equation*}
\varphi _{0}=\mathcal{F}^{-1}\left[ 1-\sum_{k=1}^{\infty }\phi \left(
N^{-k}\cdot \right) \right] .
\end{equation*}%
Let $\phi _{0}\left( \xi \right) =\mathcal{F}\varphi _{0}\left( \xi \right) ,%
\tilde{\phi}_{0}\left( \xi \right) =\mathcal{F}\varphi _{0}\left( \xi
\right) +\mathcal{F\varphi }_{1}\left( \xi \right) ,\xi \in \mathbf{R}^{d}%
\mathbf{,}\tilde{\varphi}=\mathcal{F}^{-1}\tilde{\phi},\varphi =\mathcal{F}%
^{-1}\phi ,$ and 
\begin{equation*}
\tilde{\varphi}_{k}=\sum_{l=-1}^{1}\varphi _{k+l},k\geq 1,\tilde{\varphi}%
_{0}=\varphi _{0}+\varphi _{1}
\end{equation*}%
that is 
\begin{eqnarray*}
\mathcal{F\tilde{\varphi}}_{k} &=&\phi \left( N^{-k+1}\xi \right) +\phi
\left( N^{-k}\xi \right) +\phi \left( N^{-k-1}\xi \right) \\
&=&\tilde{\phi}\left( N^{-k}\xi \right) ,\xi \in \mathbf{R}^{d},k\geq 1.
\end{eqnarray*}%
Note that $\varphi _{k}=\tilde{\varphi}_{k}\ast \varphi _{k},k\geq 0$.
Obviously, $f=\sum_{k=0}^{\infty }f\ast \varphi _{k}$ in $\mathcal{S}%
^{\prime }\left( \mathbf{R}^{d}\right) $ for $f\in \mathcal{S}\left( \mathbf{%
R}^{d}\right) .$
\end{remark}

Let $s\in \mathbf{R}$ and $p,q\geq 1$. For $\mu \in ${$\mathfrak{A}$}$%
^{\sigma }$, we introduce the Besov space $B_{pq}^{s}=B_{pq}^{\mu ,N;s}(%
\mathbf{R}^{d},V)$ as the closure of $\mathcal{S}\left( \mathbf{R}%
^{d},V\right) $ in the norm 
\begin{equation*}
|v|_{B_{pq}^{s}(\mathbf{R}^{d},V)}=|v|_{B_{pq}^{\mu ,N;s}(\mathbf{R}%
^{d},V)}=\left( \sum_{j=0}^{\infty }|J^{s}\varphi _{j}\ast v|_{L_{p}\left( 
\mathbf{R}^{d},V\right) }^{q}\right) ^{1/q},
\end{equation*}%
where $J=J_{\mu }=I-L^{\mu _{sym}}.$

We introduce the corresponding spaces of generalized functions on $%
E=[0,T]\times\mathbf{R}^{d}$ $.$ The spaces $B_{pq}^{\mu,N;s}(E,V)$ (resp. $%
H_{p}^{\mu;s}(E,V)$) consist of all measurable $B_{pq}^{\mu,N;s}(\mathbf{R}%
^{d},V)$ (resp. $H_{p}^{\mu;s}(\mathbf{R}^{d},V)$) -valued functions $f$ on $%
[0,T]$ with finite corresponding norms: 
\begin{eqnarray}
|f|_{B_{pq}^{s}(E,V)} & = &
|f|_{B_{pq}^{\mu,N;s}(E,V)}=\left(\int_{0}^{T}|f(t,\cdot)|_{B_{pq}^{\mu,N;s}(%
\mathbf{R}^{d},V)}^{q}dt\right)^{1/q},  \notag \\
|f|_{H_{p}^{s}(E,V)} & = &
|f|_{H_{p}^{\mu;s}(E,V)}=\left(\int_{0}^{T}|f(t,\cdot)|_{H_{p}^{\mu,s}(%
\mathbf{R}^{d},V)}^{p}dt\right)^{1/p}.  \label{norm11}
\end{eqnarray}

Similarly we introduce the corresponding spaces of random generalized
functions.

Let $\left(\Omega,\mathcal{F},\mathbf{P}\right)$ be a complete probability
spaces with a filtration of $\sigma-$algebras $\mathbb{F}=\left(\mathcal{F}%
_{t}\right)$ satisfying the usual conditions. Let $\mathcal{R}\left(\mathbb{F%
}\right)$ be the progressive $\sigma-$algebra on $\left[0,\infty\right)%
\times\Omega$.

The spaces $\mathbb{B}_{pp}^{s}\left(\mathbf{R}^{d},V\right)$ and $\mathbb{H}%
_{p}^{s}\left(\mathbf{R}^{d},V\right)$ consists of all $\mathcal{F}-$%
measurable random functions $f$ with values in $B_{pp}^{s}\left(\mathbf{R}%
^{d},V\right)$ and $H_{p}^{s}\left(\mathbf{R}^{d},V\right)$ with finite
norms 
\begin{equation*}
\left|f\right|_{\mathbb{B}_{pp}^{s}\left(\mathbf{R}^{d},V\right)}=\left\{ 
\mathbf{E}\left|f\right|_{B_{pp}^{s}\left(\mathbf{R}^{d},V\right)}^{p}\right%
\} ^{1/p}
\end{equation*}

and 
\begin{equation*}
\left|f\right|_{\mathbb{H}_{p}^{s}\left(\mathbf{R}^{d},V\right)}=\left\{ 
\mathbf{E}\left|f\right|_{H_{p}^{s}\left(\mathbf{R}^{d},V\right)}^{p}\right%
\} ^{1/p}.
\end{equation*}

The spaces $\mathbb{B}_{pp}^{s}\left(E,V\right)$ and $\mathbb{H}%
_{p}^{s}\left(E,V\right)$ consist of all $\mathcal{R}\left(\mathbb{F}%
\right)- $measurable random functions with values in $B_{pp}^{s}\left(E,V%
\right)$ and $H_{p}^{s}\left(E,V\right)$ with finite norms 
\begin{equation*}
\left|f\right|_{\mathbb{B}_{pp}^{s}}\left(E,V\right)=\left\{ \mathbf{E}%
\left|f\right|_{B_{pp}^{s}\left(E,V\right)}^{p}\right\} ^{1/p}
\end{equation*}

and 
\begin{equation*}
\left\vert f\right\vert _{\mathbb{H}_{p}^{s}}\left( E,V\right) =\left\{ 
\mathbf{E}\left\vert f\right\vert _{H_{p}^{s}\left( E,V\right) }^{p}\right\}
^{1/p}.
\end{equation*}

If $V_{r}=L_{r}\left( U,\mathcal{U},\Pi \right) ,r\geq 1$, the space of $r-$%
integrable measurable functions on $U$, and $V_{0}=\mathbf{R}$, we write 
\begin{eqnarray*}
B_{r,pp}^{s}\left( A\right) &=&B_{pp}^{s}\left( A,V\right) ,\hspace{1em}%
\mathbb{B}_{r,pp}^{s}\left( A\right) =\mathbb{\mathbb{B}}_{pp}^{s}\left(
A,V\right) , \\
H_{r,p}^{s}\left( A\right) &=&H_{p}^{s}\left( A,V\right) ,\hspace{1em}%
\mathbb{H}_{r,p}^{s}\left( A\right) =\mathbb{H}_{p}^{s}\left( A,V\right) ,
\end{eqnarray*}%
and%
\begin{equation*}
L_{r,p}\left( A\right) =H_{r,p}^{0}\left( A\right) ,\mathbb{L}_{r,p}\left(
A\right) =\mathbb{H}_{r,p}^{0}\left( A\right) ,
\end{equation*}%
where $A=\mathbf{R}^{d}$ or $E$. For scalar functions we drop $V$ in the
notation of function spaces.

\subsection{Main Results}

We introduce an auxiliary Levy measure $\mu^{0}$ on $\mathbf{R}_{0}^{d}$
such that the following assumption holds.

\textbf{Assumption A}$_{0}\left( \sigma \right) $. \emph{Let} $\mu ^{0}\in {%
\mathfrak{A}}=\cup _{\sigma \in \left( 0,2\right) }{\mathfrak{A}}^{\sigma }{,%
}\chi _{\left\{ \left\vert y\right\vert \leq 1\right\} }\mu ^{0}\left(
dy\right) =\mu ^{0}\left( dy\right) $, \emph{and} 
\begin{equation*}
\int \left\vert y\right\vert ^{2}\mu ^{0}\left( dy\right) +\int \left\vert
\xi \right\vert ^{4}[1+\lambda \left( \xi \right) ]^{d+3}\exp \left\{ -\psi
_{0}\left( \xi \right) \right\} d\xi \leq n_{0},
\end{equation*}

\emph{where} 
\begin{eqnarray*}
\psi_{0}\left(\xi\right) & = & \int_{\left\vert y\right\vert \leq1}\left[%
1-\cos\left(2\pi\xi\cdot y\right)\right]\mu^{0}\left(dy\right), \\
\lambda\left(\xi\right) & = & \int_{\left\vert y\right\vert
\leq1}\chi_{\sigma}\left(y\right)\left\vert y\right\vert [\left(\left\vert
\xi\right\vert \left\vert y\right\vert
\right)\wedge1]\mu^{0}\left(dy\right),\xi\in\mathbf{R}^{d}.
\end{eqnarray*}
\emph{In addition, we assume that for any} $\xi\in S_{d-1}=\left\{ \xi\in%
\mathbf{R}^{d}:\left\vert \xi\right\vert =1\right\} ,$ 
\begin{equation*}
\int_{\left\vert y\right\vert \leq1}\left\vert \xi\cdot y\right\vert
^{2}\mu^{0}\left(dy\right)\geq c_{1}>0.
\end{equation*}

For $\pi \in ${$\mathfrak{A}$}$=\cup _{\sigma \in \left( 0,2\right) }${$%
\mathfrak{A}$}$^{\sigma }$ and $R>0$, we denote 
\begin{equation*}
\pi _{R}\left( \Gamma \right) =\int \chi _{\Gamma }\left( y/R\right) \pi
\left( dy\right) ,\Gamma \in \mathcal{B}\left( \mathbf{R}_{0}^{d}\right) .
\end{equation*}

\begin{definition}
We say that a continuous function $\kappa:(0,\infty)\rightarrow(0,\infty)$
is a scaling function if $\lim_{R\rightarrow0}\kappa\left(R\right)=0,\lim_{R%
\rightarrow\infty}\kappa\left(R\right)=\infty$ and there is a nondecreasing
continuous function $l\left(\varepsilon\right),\varepsilon>0,$ such that $%
\lim_{\varepsilon\rightarrow0}l\left(\varepsilon\right)=0$ and 
\begin{equation*}
\kappa\left(\varepsilon r\right)\leq
l\left(\varepsilon\right)\kappa(r),r>0,\varepsilon>0.
\end{equation*}
We call $l\left(\varepsilon\right),\varepsilon>0,$ a scaling factor of $%
\kappa$.
\end{definition}

For a scaling function $\kappa $ with a scaling factor $l$ and $\pi \in ${$%
\mathfrak{A}$}$^{\sigma }$ we introduce the following assumptions.

\textbf{D}$\left(\kappa,l\right)$\textbf{. }\emph{For every} $R>0,$ 
\begin{equation*}
\tilde{\pi}_{R}\left(dy\right)=\kappa\left(R\right)\pi_{R}\left(dy\right)%
\geq1_{\left\{ \left\vert y\right\vert \leq1\right\} }\mu^{0}\left(dy\right),
\end{equation*}
\emph{with }$\mu^{0}=\mu^{0;\pi}$\emph{\ satisfying Assumption }\textbf{A}$%
_{0}\left(\sigma\right)$. \emph{If }$\sigma=1$ \emph{we, in addition assume
that }$\int_{R<\left\vert y\right\vert \leq
R^{\prime}}y\mu^{0}\left(dy\right)=0$ \emph{for any} $0<R<R^{\prime}\leq1.$ 
\emph{\ Here }$\tilde{\pi}_{R}\left(dy\right)=\kappa\left(R\right)\pi_{R}%
\left(dy\right).$

\textbf{B}$\left(\kappa,l\right)$\textbf{. }\emph{There exist }$\alpha_{1}$%
\emph{\ and }$\alpha_{2}$\emph{\ and a constant }$N_{0}>0$ \emph{such that} 
\begin{equation*}
\int_{\left\vert z\right\vert \leq1}\left\vert z\right\vert ^{\alpha_{1}}%
\tilde{\pi}_{R}(dz)+\int_{\left\vert z\right\vert >1}\left\vert z\right\vert
^{\alpha_{2}}\tilde{\pi}_{R}(dz)\leq N_{0}\text{ }\forall R>0,
\end{equation*}
\emph{where }$\alpha_{1},\alpha_{2}\in(0,1]\text{ \emph{if} }\sigma\in(0,1)%
\text{; }\alpha_{1},\alpha_{2}\in(1,2]\text{ \emph{if} }\sigma\in(1,2)$\emph{%
; }$\alpha_{1}\in(1,2]$\emph{\ and }$\alpha_{2}\in\left[0,1\right)$\emph{\
if }$\sigma=1$\emph{.}

The main result for (\ref{mainEq}) is the following statement.

\begin{theorem}
\label{t1}Let $\pi ,\mu \in \mathfrak{A}^{\sigma },p\in \left( 1,\infty
\right) ,s\in \mathbf{R}$. Assume there is a scaling function $\kappa $ with
a scaling factor $l$ such that \textbf{D}$\left( \kappa ,l\right) $ and 
\textbf{B}$\left( \kappa ,l\right) $ hold for both, $\pi $ and $\mu $.
Assume 
\begin{equation*}
\int_{1}^{\infty }\frac{dt}{t\gamma \left( t\right) ^{1\wedge \alpha _{2}}}%
<\infty ,
\end{equation*}%
and there are $\beta _{0}<\alpha _{2}$ and $\beta _{1},\beta _{2}>0$ such
that 
\begin{equation*}
\int_{0}^{1}\gamma \left( t\right) ^{-\beta _{1}}dt+\int_{0}^{1}l\left(
t\right) ^{\beta _{2}}\frac{dt}{t}+\int_{1}^{\infty }\frac{1}{\gamma \left(
t\right) ^{\beta _{0}}}\frac{dt}{t}<\infty \text{ if }p>2,
\end{equation*}%
where $\gamma \left( t\right) =\inf \left\{ r:l\left( r\right) \geq
t\right\} ,t>0.$

Then for each $f\in \mathbb{H}_{p}^{\mu ;s}(E),g\in \mathbb{B}_{pp}^{\mu
,N;s+1-1/p}\left( \mathbf{R}^{d}\right) $, $\Phi \in \mathbb{B}_{p,pp}^{\mu
,N;s+1-1/p}\left( E\right) \cap \mathbb{H}_{2,p}^{\mu ;s+1/2}\left( E\right) 
\hspace{1em}$if $p\in \left[ 2,\infty \right) $ and $\Phi \in \mathbb{B}%
_{p,pp}^{\mu ,N;s+1-1/p}\left( E\right) $ if\hspace{1em}$p\in \left(
1,2\right) $, there is a unique $u\in \mathbb{H}_{p}^{\mu ;s+1}\left(
E\right) $ solving (\ref{mainEq}). Moreover, there is $C=C\left( d,p,\kappa
,l,n_{0},N_{0},c_{1}\right) $ such that for $p\in \left[ 2,\infty \right) $, 
\begin{eqnarray*}
&&\left\vert L^{\mu }u\right\vert _{\mathbb{H}_{p}^{\mu ;s}\left( E\right) }
\\
&\leq &C\left[ \left\vert f\right\vert _{\mathbb{H}_{p}^{\mu ;s}\left(
E\right) }+\left\vert g\right\vert _{\mathbb{B}_{pp}^{\mu ,N;s+1-1/p}\left( 
\mathbf{R}^{d}\right) }+\left\vert \Phi \right\vert _{\mathbb{B}_{p,pp}^{\mu
,N;s+1-1/p}\left( E\right) }+\left\vert \Phi \right\vert _{\mathbb{H}%
_{2,p}^{\mu ;s+1/2}\left( E\right) }\right] , \\
&&\left\vert u\right\vert _{\mathbb{H}_{p}^{\mu ;s}\left( E\right) } \\
&\leq &C[\rho _{\lambda }\left\vert f\right\vert _{\mathbb{H}_{p}^{\mu
;s}\left( E\right) }+\rho _{\lambda }^{1/p}\left\vert g\right\vert _{\mathbb{%
H}_{p}^{\mu ;s}\left( \mathbf{R}^{d}\right) }+\rho _{\lambda
}^{1/p}\left\vert \Phi \right\vert _{\mathbb{H}_{p,p}^{\mu ;s}\left( \mathbf{%
R}^{d}\right) }+\rho _{\lambda }^{1/2}\left\vert \Phi \right\vert _{\mathbb{H%
}_{2,p}^{\mu ;s}\left( \mathbf{R}^{d}\right) }],
\end{eqnarray*}

and for $p\in \left( 1,2\right) $, 
\begin{eqnarray*}
\left\vert L^{\mu }u\right\vert _{\mathbb{H}_{p}^{\mu ;s}\left( E\right) }
&\leq &C\left[ \left\vert f\right\vert _{\mathbb{H}_{p}^{\mu ;s}\left(
E\right) }+\left\vert g\right\vert _{\mathbb{B}_{pp}^{\mu ,N;s+1-1/p}\left( 
\mathbf{R}^{d}\right) }+\left\vert \Phi \right\vert _{\mathbb{B}_{p,pp}^{\mu
,N;s+1-1/p}\left( E\right) }\right] , \\
\left\vert u\right\vert _{\mathbb{H}_{p}^{\mu ;s}\left( E\right) } &\leq &C%
\left[ \rho _{\lambda }\left\vert f\right\vert _{\mathbb{H}_{p}^{\mu
;s}\left( E\right) }+\rho _{\lambda }^{1/p}\left\vert g\right\vert _{\mathbb{%
H}_{p}^{\mu ;s}\left( \mathbf{R}^{d}\right) }+\rho _{\lambda
}^{1/p}\left\vert \Phi \right\vert _{\mathbb{H}_{p,p}^{\mu ;s}\left( \mathbf{%
R}^{d}\right) }\right] ,
\end{eqnarray*}

where $\rho_{\lambda}=\frac{1}{\lambda}\wedge T.$
\end{theorem}

\begin{remark}
\label{re0}1. Assumptions \textbf{D}$\left( \kappa ,l\right) ,$ \textbf{B}$%
\left( \kappa ,l\right) $ hold for both, $\pi ,\mu $, means that $\kappa ,l,$
and the parameters $\alpha _{1},\alpha _{2},n_{0},c_{1},N_{0}$ are the same (%
$\mu ^{0}$ could be different).

2. For every $\varepsilon >0$, $B_{pp}^{\mu ,N;s+\varepsilon }\left( \mathbf{%
R}^{d}\right) $ is continuously embedded into $H_{p}^{\mu ;s}\left( \mathbf{R%
}^{d}\right) ,p>1$; for $p\geq 2,$ $H_{p}^{\mu ;s}\left( \mathbf{R}%
^{d}\right) $\text{ is continuously embedded into $B_{pp}^{\mu ,N;s}\left( 
\mathbf{R}^{d}\right) .$ }
\end{remark}

\subsection{Examples}

Let $\Lambda \left( dt\right) $ be a measure on $\left( 0,\infty \right) $
such that $\int_{0}^{\infty }\left( 1\wedge t\right) \Lambda \left(
dt\right) <\infty $, and let 
\begin{equation*}
\phi \left( r\right) =\int_{0}^{\infty }\left( 1-e^{-rt}\right) \Lambda
\left( dt\right) ,r\geq 0,
\end{equation*}%
be a Bernstein function (see \cite{vo}, \cite{KK1}). Let 
\begin{equation*}
j\left( r\right) =\int_{0}^{\infty }\left( 4\pi t\right) ^{-\frac{d}{2}}\exp
\left( -\frac{r^{2}}{4t}\right) \Lambda \left( dt\right) ,r>0.
\end{equation*}%
We consider $\pi \in ${$\mathfrak{A}$}$=\cup _{\sigma \in \left( 0,2\right)
} ${$\mathfrak{A}$}$^{\sigma }\,\ $defined in radial and angular coordinates 
$r=\left\vert y\right\vert ,w=y/r,$ as 
\begin{equation}
\pi \left( \Gamma \right) =\int_{0}^{\infty }\int_{\left\vert w\right\vert
=1}\chi _{\Gamma }\left( rw\right) a\left( r,w\right) j\left( r\right)
r^{d-1}S\left( dw\right) dr,\Gamma \in \mathcal{B}\left( \mathbf{R}%
_{0}^{d}\right) ,  \label{fe1}
\end{equation}%
where $S\left( dw\right) $ is a finite measure on the unite sphere on $%
\mathbf{R}^{d}$. If $S\left( dw\right) =dw$ is the Lebesgue measure on the
unit sphere, then 
\begin{equation*}
\pi \left( \Gamma \right) =\pi ^{J,a}\left( \Gamma \right) =\int_{\mathbf{R}%
^{d}}\chi _{\Gamma }\left( y\right) a\left( \left\vert y\right\vert
,y/\left\vert y\right\vert \right) J\left( y\right) dy,\Gamma \in \mathcal{B}%
\left( \mathbf{R}_{0}^{d}\right) ,
\end{equation*}%
where $J\left( y\right) =j\left( \left\vert y\right\vert \right) ,y\in 
\mathbf{R}^{d}.$ Let $\mu =\pi ^{J,1},$ i.e., 
\begin{equation}
\mu \left( \Gamma \right) =\int_{\mathbf{R}^{d}}\chi _{\Gamma }\left(
y\right) J\left( y\right) dy,\Gamma \in \mathcal{B}\left( \mathbf{R}%
_{0}^{d}\right) .  \label{fe2}
\end{equation}

We assume

\textbf{H. }(i) \emph{There is} $N>0$ \emph{so that} 
\begin{equation*}
N^{-1}\phi\left(r^{-2}\right)r^{-d}\leq j\left(r\right)\leq
N\phi\left(r^{-2}\right)r^{-d},r>0.
\end{equation*}

(ii) \emph{There are} $0<\delta_{1}\leq\delta_{2}\leq1$ \emph{and }$N>0$%
\emph{\ so that for} $0<r\leq R$ 
\begin{equation*}
N^{-1}\left(\frac{R}{r}\right)^{\delta_{1}}\leq\frac{\phi\left(R\right)}{%
\phi\left(r\right)}\leq N\left(\frac{R}{r}\right)^{\delta_{2}}.
\end{equation*}

\textbf{G. }\emph{There is} $\rho_{0}\left(w\right)\geq0,\left\vert
w\right\vert =1,$ \emph{such that} $\rho_{0}\left(w\right)\leq
a\left(r,w\right)\leq1,r>0,\left\vert w\right\vert =1,$ \emph{and for all} $%
\left\vert \xi\right\vert =1,$ 
\begin{equation*}
\int_{\left\vert w\right\vert =1}\left\vert \xi\cdot w\right\vert
^{2}\rho_{0}\left(w\right)S\left(dw\right)\geq c>0
\end{equation*}
\emph{for some} $c>0$.

For example, in \cite{vo} and \cite{KK1} among others the following specific
Bernstein functions satisfying \textbf{H} are listed:

(0) $\phi\left(r\right)=\sum_{i=1}^{n}r^{\alpha_{i}},\alpha_{i}\in\left(0,1%
\right),i=1,\ldots,n;$

(1) $\phi\left(r\right)=\left(r+r^{\alpha}\right)^{\beta},\alpha,\beta\in%
\left(0,1\right);$

(2) $\phi\left(r\right)=r^{\alpha}\left(\ln\left(1+r\right)\right)^{\beta},%
\alpha\in\left(0,1\right),\beta\in\left(0,1-\alpha\right);$

(3) $\phi\left(r\right)=\left[\ln\left(\cosh\sqrt{r}\right)\right]%
^{\alpha},\alpha\in\left(0,1\right).$

All the assumptions of Theorem \ref{t1} hold under \textbf{H, G.}

Indeed, \textbf{H} implies that there are $0<c\leq C$ so that 
\begin{eqnarray*}
cr^{-d-2\delta _{1}} &\leq &j\left( r\right) \leq Cr^{-d-2\delta _{2}},r\leq
1, \\
cr^{-d-2\delta _{2}} &\leq &j\left( r\right) \leq Cr^{-d-2\delta _{1}},r>1.
\end{eqnarray*}%
Hence $2\delta _{1}\leq \sigma \leq 2\delta _{2}.$ In this case $\kappa
\left( R\right) =j\left( R\right) ^{-1}R^{-d},R>0,$ is a scaling function,
and $\kappa \left( \varepsilon R\right) \leq l\left( \varepsilon \right)
\kappa \left( R\right) ,\varepsilon ,R>0,$ with 
\begin{equation*}
l\left( \varepsilon \right) =\left\{ 
\begin{array}{cc}
C_{1}\varepsilon ^{2\delta _{1}} & \text{if }\varepsilon \leq 1, \\ 
C_{1}\varepsilon ^{2\delta _{2}} & \text{if }\varepsilon >1%
\end{array}%
\right.
\end{equation*}%
for some $C_{1}>0$. Hence 
\begin{equation*}
\gamma \left( t\right) =l^{-1}\left( t\right) =\left\{ 
\begin{array}{c}
C_{1}^{-1/2\delta _{1}}t^{1/2\delta _{1}}\text{ if }t\leq C_{1}, \\ 
C_{1}^{-1/2\delta _{2}}t^{1/2\delta _{2}}\text{ if }t>C_{1}.%
\end{array}%
\right.
\end{equation*}%
We see easily that $\alpha _{1}$ is any number $>2\delta _{2}$ and $\alpha
_{2}$ is any number $<2\delta _{1}.$ The measure $\mu ^{0}$ for $\pi $ is 
\begin{equation*}
\mu ^{0}\left( dy\right) =\mu ^{0,\pi }\left( dy\right) =c_{1}\int \chi
_{dy}\left( rw\right) \chi _{\left\{ r\leq 1\right\} }r^{-1-2\delta
_{1}}\rho _{0}\left( w\right) S\left( dw\right) dr;
\end{equation*}%
and $\mu ^{0}$ for $\mu $ is 
\begin{equation*}
\mu ^{0}\left( dy\right) =\mu ^{0,\mu }\left( dy\right) =c_{1}^{\prime }\int
\chi _{dy}\left( rw\right) \chi _{\left\{ r\leq 1\right\} }r^{-1-2\delta
_{1}}dwdr
\end{equation*}%
with some $c_{1},c_{1}^{\prime }.$ Integrability conditions easily follow.

\section{Auxiliary results}

\subsection{Approximation of input functions}

Let $V_{r}=L_{r}\left( U,\mathcal{U},\Pi \right) ,r\geq 1$, the space of $r-$%
integrable measurable functions on $U$, and $V_{0}=\mathbf{R}$. For brevity
of notation we write 
\begin{equation*}
B_{r,pp}^{s}\left( A\right) =B_{pp}^{s}\left( A;V_{r}\right) ,\hspace{1em}%
\mathbb{B}_{r,pp}^{s}\left( A\right) =\mathbb{\mathbb{B}}_{pp}^{s}\left(
A;V_{r}\right) ,
\end{equation*}%
\begin{eqnarray*}
H_{r,p}^{s}\left( A\right) &=&H_{p}^{s}\left( A;V_{r}\right) ,\hspace{1em}%
\mathbb{H}_{r,p}^{s}\left( A\right) =\mathbb{H}_{p}^{s}\left( A;V_{r}\right)
, \\
L_{r,p}\left( A\right) &=&H_{r,p}^{0}\left( A\right) ,\mathbb{L}_{r,p}\left(
A\right) =\mathbb{H}_{r,p}^{0}\left( A\right) ,
\end{eqnarray*}%
where $A=\mathbf{R}^{d}$ or $E$. We use the following equivalent norms of
Besov spaces $B_{r,pp}^{s}\left( \mathbf{R}^{d}\right) ,r=0,p,$ (see \cite%
{Mph2})%
\begin{equation*}
|v|_{\tilde{B}_{r,pp}^{s}(\mathbf{R}^{d})}=\left( \sum_{j=0}^{\infty }\kappa
\left( N^{-j}\right) ^{-sp}\int |\varphi _{j}\ast v|_{V_{r}\ }^{p}dx\right)
^{1/p},
\end{equation*}%
where $\varphi _{j}=\varphi _{j}^{N},j\geq 0,$ is the system of functions
defined in Remark \ref{re1}. The equivalent norms of $H_{r,p}^{s}\left( 
\mathbf{R}^{d}\right) ,r=0,2,$ (see \cite{Mph2}) are defined by%
\begin{equation}
|v|_{\tilde{H}_{r,p}^{s}(\mathbf{R}^{d})}=|v|_{\tilde{H}_{r,p}^{\kappa ,N;s}(%
\mathbf{R}^{d})}=\left\vert \left( \sum_{j=0}^{\infty }\left\vert \kappa
\left( N^{-j}\right) ^{-s}\varphi _{j}\ast v\right\vert _{V_{r}}^{2}\right)
^{1/2}\right\vert _{L_{p}\left( \mathbf{R}^{d}\right) }.
\end{equation}%
We define the equivalent norms of functions on $E$ as well:%
\begin{equation*}
|v|_{\tilde{B}_{p,pp}^{s}(E)}=\left( \int_{0}^{T}\left\vert v\left( t\right)
\right\vert _{\tilde{B}_{p,pp}^{s}(\mathbf{R}^{d})}^{p}dt\right) ^{1/p},|v|_{%
\tilde{H}_{2,p}^{s}(E)}=|v|_{\tilde{H}_{2,p}^{\kappa ,N;s}(E)}=\left(
\int_{0}^{T}\left\vert v\left( t\right) \right\vert _{\tilde{H}_{2,pp}^{s}(%
\mathbf{R}^{d})}^{p}dt\right) ^{1/p}.
\end{equation*}%
For $D=D_{r}\left( A\right) =B_{r,pp}^{s}\left( A\right) $ or $%
H_{r,p}^{s}\left( A\right) ,A=\mathbf{R}^{d},E$, we consider corresponding
equivalent norms on random function spaces $\mathbb{D}=\mathbb{D}_{r}=%
\mathbb{B}_{r,pp}^{s}\left( A\right) $ or $\mathbb{H}_{r,p}^{s}\left(
A\right) :$%
\begin{equation*}
\left\vert v\right\vert _{\mathbb{\tilde{D}}}=\left\{ \mathbf{E}\left(
\left\vert v\right\vert _{\tilde{D}}^{p}\right) \right\} ^{1/p}.
\end{equation*}

Let $U_{n}\in \mathcal{U},U_{n}\subseteq U_{n+1},n\geq 1,\cup _{n}U_{n}=U$
and $\pi \left( U_{n}\right) <\infty ,n\geq 1.$ We denote by $\mathbb{\tilde{%
C}}_{r.p}^{\infty }\left( E\right) ,1\leq p<\infty ,$ the space of all $%
\mathcal{R}\left( \mathbb{F}\right) \otimes \mathcal{B}\left( \mathbf{R}%
^{d}\right) -$measurable $V_{r}$ -valued random functions $\Phi $ on $E$
such that for every$\gamma \in \mathbf{N}_{0}^{d}$, 
\begin{equation*}
\mathbf{E}\int_{0}^{T}\sup_{x\in \mathbf{R}^{d}}\left\vert D^{\gamma }\Phi
\left( t,x\right) \right\vert _{V_{r}}^{p}dt+\mathbf{E}\left[ \left\vert
D^{\gamma }\Phi \right\vert _{L_{p}\left( E;V_{r}\right) }^{p}\right]
<\infty ,
\end{equation*}%
and $\Phi =\Phi \chi _{U_{n}}$ for some $n$ if $r=2,p.$ Similarly we define
the space $\tilde{\mathbb{C}_{r,p}^{\infty }}\left( \mathbf{R}^{d}\right) $
by replacing $\mathcal{R}\left( \mathbb{F}\right) $ and $E$ by $\mathcal{F}$
and $\mathbf{R}^{d}$ respectively in the definition of $\mathbb{\tilde{C}}%
_{r,p}^{\infty }\left( E\right) $.

\begin{lemma}
\label{lem-approximation1} Let $D\left( \kappa ,l\right) $ and $B\left(
\kappa ,l\right) $ hold for $\mu \in ${$\mathfrak{A}$}$^{\sigma }$with
scaling function $\kappa $ and scaling factor $l$. Let $U_{n}\in \mathcal{U}%
,U_{n}\subseteq U_{n+1},n\geq 1,\cup _{n}U_{n}=U$ and $\pi \left(
U_{n}\right) <\infty ,n\geq 1.$ Let $s\in \mathbf{R},p\in \left( 1,\infty
\right) $, $\Phi \in \mathbb{D}_{r,p}$, where $\mathbb{D}_{r,p}=\mathbb{D}%
_{r,p}\left( A\right) =\mathbb{B}_{r,pp}^{s}\left( A\right) $ with $r=0,p,$
or $\mathbb{D}_{r,p}=\mathbb{H}_{r,p}^{s}\left( A\right) $ with $r=0,2$, $A=%
\mathbf{R}^{d}$ or $E$. For $\Phi \in \mathbb{D}_{r,p}$ we set%
\begin{equation*}
\Phi _{n}=\sum_{j=0}^{n}\Phi \ast \varphi _{j}\chi _{U_{n}},\text{ if }%
r=2,p,~\Phi _{n}=\sum_{j=0}^{n}\Phi \ast \varphi _{j},\text{ if }r=0.
\end{equation*}%
Then there is $C>0$ so that%
\begin{equation*}
\left\vert \Phi _{n}\right\vert _{\mathbb{D}_{r,p}}\leq C\left\vert \Phi
\right\vert _{\mathbb{D}_{r,p}},\Phi \in \mathbb{D}_{r,p},n\geq 1\text{,}
\end{equation*}%
and $\left\vert \Phi _{n}-\Phi \right\vert _{\mathbb{D}_{r,p}}\rightarrow 0$
as $n\rightarrow \infty .$ Moreover, for $r=0,2,p,$ every $n$ and multiindex 
$\gamma \in \mathbf{N}_{0}^{d}$,%
\begin{eqnarray*}
\mathbf{E}\int_{0}^{T}\sup_{x}\left\vert D^{\gamma }\Phi _{n}\right\vert
_{V_{r}}^{p}dt+\left\vert D^{\gamma }\Phi _{n}\right\vert _{\mathbb{L}%
_{r,p}\left( E\right) }^{p} &<&\infty \text{ if }A=E, \\
\mathbf{E[}\sup_{x}\left\vert D^{\gamma }\Phi _{n}\right\vert
_{V_{r}}^{p}]+\left\vert D^{\gamma }\Phi _{n}\right\vert _{\mathbb{L}%
_{r,p}\left( \mathbf{R}^{d}\right) }^{p} &<&\infty \text{ if }A=\mathbf{R}%
^{d},
\end{eqnarray*}%
where
\end{lemma}

\begin{proof}
Let $\tilde{\Phi}_{n}=\Phi \chi _{U_{n}},n\geq 1.$ Since%
\begin{equation*}
\varphi _{k}=\sum_{l=-1}^{1}\varphi _{k+l}\ast \varphi _{k},k\geq 1,\varphi
_{0}=\left( \varphi _{0}+\varphi _{1}\right) \ast \varphi _{0},
\end{equation*}%
we have for $n>1,$%
\begin{eqnarray*}
\left( \tilde{\Phi}_{n}-\Phi _{n}\right) \ast \varphi _{k} &=&0,k<n, \\
\left( \tilde{\Phi}_{n}-\Phi _{n}\right) \ast \varphi _{k} &=&\left( \tilde{%
\Phi}_{n}\ast \varphi _{k-1}+\tilde{\Phi}_{n}\ast \varphi _{k}+\tilde{\Phi}%
_{n}\ast \varphi _{k+1}\right) \ast \varphi _{k},k>n+1, \\
\left( \tilde{\Phi}_{n}-\Phi _{n}\right) \ast \varphi _{n} &=&\left( \tilde{%
\Phi}_{n}\ast \varphi _{n+1}\right) \ast \varphi _{n}, \\
\left( \tilde{\Phi}_{n}-\Phi _{n}\right) \ast \varphi _{n+1} &=&\left( 
\tilde{\Phi}_{n}\ast \varphi _{n+1}+\tilde{\Phi}_{n}\ast \varphi
_{n+2}\right) \ast \varphi _{n+1}.
\end{eqnarray*}%
Let $V_{r}=L_{r}\left( U,\mathcal{U},\Pi \right) ,r=2,p$. By Corollary 2 in 
\cite{Mph2}, there is a constant $C$ independent of $\Phi \in \mathbb{H}%
_{2,p}^{s}\left( E\right) $ so that%
\begin{eqnarray*}
&&\left\vert \left( \sum_{j=0}^{\infty }\left\vert \kappa \left(
N^{-j}\right) ^{-s}\left( \tilde{\Phi}_{n}-\Phi _{n}\right) \ast \varphi
_{j}\right\vert _{V_{2}}^{2}\right) ^{1/2}\right\vert _{L_{p}\left( E\right)
} \\
&\leq &C\left\vert \left( \sum_{j=n}^{\infty }\left\vert \kappa \left(
N^{-j}\right) ^{-s}\Phi \ast \varphi _{j}\right\vert _{V_{2}}^{2}\right)
^{1/2}\right\vert _{L_{p}\left( E\right) }\rightarrow 0
\end{eqnarray*}%
as $n\rightarrow \infty .$ Obviously 
\begin{eqnarray*}
\left\vert \left( \tilde{\Phi}_{n}-\Phi _{n}\right) \ast \varphi
_{j}\right\vert _{L_{r,p}\left( E\right) } &\leq
&C\sum_{k=j-1}^{j+1}\left\vert \tilde{\Phi}_{n}\ast \varphi _{k}\right\vert
_{L_{r,p}\left( E\right) },j\geq n, \\
\left\vert \left( \tilde{\Phi}_{n}-\Phi _{n}\right) \ast \varphi
_{j}\right\vert _{L_{r,p}\left( E\right) } &=&0,j<n,r=0,p,
\end{eqnarray*}%
and 
\begin{equation*}
\left\vert \Phi _{n}\right\vert _{\mathbb{D}_{r,p}\left( E\right) }\leq
C\left\vert \Phi \right\vert _{\mathbb{D}_{r,p}\left( E\right) },\Phi \in 
\mathbb{D}_{r,p},n\geq 1,r=0,2,p\text{.}
\end{equation*}%
Thus $\left\vert \Phi _{n}-\Phi \right\vert _{\mathbb{D}_{r,p}\left(
E\right) }\rightarrow 0$ as $n\rightarrow \infty ,r=0,2,p.$

Let $r=0,2,p$, $\Phi \in \mathbb{D}_{r,p}\left( E\right) $. Obviously, for
any $k\geq 0,$%
\begin{equation*}
\mathbf{E}\int_{E}\left\vert \Phi \ast \varphi _{k}\right\vert
_{V_{r}}^{p}dxdt<\infty ,
\end{equation*}%
where $r=0,2,p$ with $V_{0}=\mathbf{R}.$. Since for any multiindex $\gamma ,$%
\begin{equation*}
\Phi \ast \varphi _{k}=\Phi \ast \varphi _{k}\ast \tilde{\varphi}%
_{k},D^{\gamma }\Phi \ast \varphi _{k}=\Phi \ast \varphi _{k}\ast D^{\gamma }%
\tilde{\varphi}_{k},
\end{equation*}%
and $\mathbf{P}$-a.s. for all $s,x,$ with $\frac{1}{q}+\frac{1}{p}=1,$%
\begin{eqnarray*}
\left\vert D^{\gamma }\Phi \ast \varphi _{k}\left( s,x\right) \right\vert
_{V_{r}} &\leq &\int \left\vert \Phi \ast \varphi _{k}\left( s,x-y\right)
\right\vert _{V_{r}}\left\vert D^{\gamma }\tilde{\varphi}_{k}\left( y\right)
\right\vert dy, \\
\sup_{x}\left\vert D^{\gamma }\Phi \ast \varphi _{k}\left( s,x\right)
\right\vert _{V_{r}} &\leq &\left( \int \left\vert \Phi \ast \varphi
_{k}\left( s,\cdot \right) \right\vert _{V_{r}}^{p}dx\right)
^{1/p}\left\vert D^{\gamma }\tilde{\varphi}_{k}\right\vert _{L_{q}\left( 
\mathbf{R}^{d}\right) },
\end{eqnarray*}%
we have for any multiindex $\gamma ,$%
\begin{equation*}
\left\vert D^{\gamma }\Phi \ast \varphi _{k}\right\vert _{\mathbb{L}%
_{r,p}\left( E\right) }<\infty ,
\end{equation*}%
and%
\begin{equation*}
\mathbf{E}\int_{0}^{T}\sup_{x}\left\vert D^{\gamma }\Phi \ast \varphi
_{k}\right\vert _{V_{r}}^{p}dt<\infty ,r=0,2,p.
\end{equation*}%
The proof for the case of $A=\mathbf{R}^{d}$ is a repeat with obvious changes%
$.$ The statement follows.
\end{proof}

\begin{corollary}
\label{cor:approx1}The space $C_{0}^{\infty }\left( \mathbf{R}%
^{d};V_{r}\right) $ of $V_{r}$-valued infinitely differentiable functions
with compact support is dense in $D_{r}\left( \mathbf{R}^{d}\right) ,r=0,2,p$%
.
\end{corollary}

\begin{proof}
In the view of Lemma \ref{lem-approximation1}, it suffices to show that for
any $V=V_{r}$-valued function $v$ such that for all multiindex $\gamma \in 
\mathbf{N}_{0}^{d}$, 
\begin{equation*}
\sup_{x}\left\vert D^{\gamma }v\left( x\right) \right\vert
_{V_{r}}+\left\vert D^{\gamma }v\right\vert _{L_{p}\left( \mathbf{R}%
^{d};V_{r}\right) }<\infty
\end{equation*}

there exists $v_{n}\in C_{0}^{\infty }\left( \mathbf{R}^{d},V_{r}\right) $
so that $v_{n}\rightarrow v$ in $D_{r}\left( \mathbf{R}^{d}\right) $. Let $%
g\in C_{0}^{\infty }\left( \mathbf{R}^{d}\right) $ with $0\leq g\left(
x\right) \leq 1,x\in \mathbf{R}^{d}$, $g\left( x\right) =1$ for $\left\vert
x\right\vert \leq 1$, and $g\left( x\right) =0$ for $\left\vert x\right\vert
\geq 2$. Let 
\begin{equation*}
v_{n}\left( x\right) :=v\left( x\right) g\left( x/n\right) ,x\in \mathbf{R}%
^{d}.
\end{equation*}%
Obviously $v_{n}\in C_{0}^{\infty }\left( \mathbf{R}^{d},V_{r}\right) $, and
for any multiindex $\beta ,$

\begin{eqnarray*}
D^{\beta }v_{n}\left( x\right) &=&D^{\beta }v\left( x\right) g\left(
x/n\right) +\sum_{\substack{ \beta _{1}+\beta _{2}=\beta ,  \\ \left\vert
\beta _{2}\right\vert \geq 1}}n^{-\left\vert \beta _{2}\right\vert }D^{\beta
_{1}}v\left( x\right) \left( D^{\beta _{2}}g\right) \left( x/n\right) ,x\in 
\mathbf{R}^{d}, \\
\left\vert D^{\beta }v_{n}\right\vert _{Lp\left( \mathbf{R}^{d};V_{r}\right)
} &\leq &C\left( \left\vert \beta \right\vert \right) \sup_{\beta ^{\prime
}\leq \left\vert \beta \right\vert }\left\vert D^{\beta ^{\prime
}}v\right\vert _{Lp\left( \mathbf{R}^{d};V_{r}\right) },
\end{eqnarray*}%
and $\left\vert D^{\beta }v_{n}-D^{\beta }v\right\vert _{Lp\left( \mathbf{R}%
^{d};V_{r}\right) }\rightarrow 0$. Since for any multiindex $\beta $ we have$%
\int y^{\beta }\varphi _{j}\left( y\right) dy=0,$ it follows for $m>0,j\geq
1 $, by Taylor remainder theorem, for $x\in \mathbf{R}^{d},$

\begin{eqnarray*}
v_{n}\ast \varphi _{j}\left( x\right) &=&\int \varphi _{j}\left( y\right)
\left\{ v_{n}\left( x-y\right) -\sum_{\beta :\left\vert \beta \right\vert
\leq m}\frac{D^{\beta }v_{n}\left( x\right) }{\beta !}\left( -y\right)
^{\beta }\right\} dy \\
&=&\int \varphi _{j}\left( y\right) \sum_{\beta :\left\vert \beta
\right\vert =m+1}\int_{0}^{1}\frac{\left( 1-t\right) ^{m+1}}{\left(
m+1\right) !}\left( D^{\beta }v_{n}\right) \left( x-ty\right) \left(
-y\right) ^{\beta }dtdy \\
&=&N^{-j(m+1)}\sum_{\beta :\left\vert \beta \right\vert =m+1}\int \varphi
\left( y\right) \int_{0}^{1}\frac{\left( 1-t\right) ^{m+1}}{\left(
m+1\right) !}\left( D^{\beta }v_{n}\right) \left( x-tN^{-j}y\right) dt\left(
-y\right) ^{\beta }dy.
\end{eqnarray*}%
By Lemma 6 of \cite{Mph2}, there exists $\sigma ^{\prime }$ such that $%
\kappa \left( N^{-j}\right) ^{-s}\leq N^{js\sigma ^{\prime }}$. Let $m>1$ be
such that $t=N^{\sigma ^{\prime }s}N^{-m}<1$. Hence there is a constant $%
C=C\left( m\right) $ (independent of $n$) so that 
\begin{equation*}
\kappa \left( N^{-j}\right) ^{-s}\left\vert v_{n}\ast \varphi
_{j}\right\vert _{L_{p}\left( \mathbf{R}^{d};V_{2}\right) }\leq C\left(
m\right) t^{j},j\geq 0.
\end{equation*}%
Now, for any $k\geq 0,$

\begin{eqnarray*}
&&\left\vert \left( \sum_{j=k}^{\infty }\left\vert \kappa \left(
N^{-j}\right) ^{-s}v_{n}\ast \varphi _{j}\left( x\right) \right\vert
_{V_{2}}^{2}\right) ^{1/2}\right\vert _{L_{p}\left( \mathbf{R}^{d}\right) }
\\
&\leq &\left\vert \sum_{j=k}^{\infty }\left\vert \kappa \left( N^{-j}\right)
^{-s}v_{n}\ast \varphi _{j}\left( x\right) \right\vert _{V_{2}}\right\vert
_{L_{p}\left( \mathbf{R}^{d}\right) }\leq \sum_{j\geq k}\left\vert \kappa
\left( N^{-j}\right) ^{-s}v_{n}\ast \varphi _{j}\left( x\right) \right\vert
_{L_{p}\left( \mathbf{R}^{d};V_{2}\right) } \\
&\leq &C\left( m\right) \sum_{j\geq k}t^{j}.
\end{eqnarray*}

Since the same estimate holds for $v,$ 
\begin{equation*}
\left\vert \left( \sum_{j\geq k}\left\vert \kappa \left( N^{-j}\right)
^{-s}v\ast \varphi _{j}\left( x\right) \right\vert _{V_{2}}^{2}\right)
^{1/2}\right\vert _{L_{p}\left( \mathbf{R}^{d}\right) }\leq C\left( m\right)
\sum_{j\geq k}t^{j},
\end{equation*}%
and%
\begin{equation*}
\left\vert \left( \sum_{j<k}\left\vert \kappa \left( N^{-j}\right)
^{-s}(v-v_{n})\ast \varphi _{j}\left( x\right) \right\vert
_{V_{2}}^{2}\right) ^{1/2}\right\vert _{L_{p}\left( \mathbf{R}^{d}\right)
}\rightarrow 0,
\end{equation*}%
it follows that 
\begin{equation*}
\left\vert \left( \sum_{j}\left\vert \kappa \left( N^{-j}\right)
^{-s}(v-v_{n})\ast \varphi _{j}\left( x\right) \right\vert
_{V_{2}}^{2}\right) ^{1/2}\right\vert _{L_{p}\left( \mathbf{R}^{d}\right)
}\rightarrow 0
\end{equation*}%
as $n\rightarrow \infty $.

Likewise, for $r=0,p,$ 
\begin{equation*}
\lim_{n\rightarrow \infty }\left\vert v_{n}-v\right\vert
_{B_{r,pp}^{s}}=\lim_{n\rightarrow \infty }\left( \sum_{j=0}^{\infty
}\left\vert \kappa \left( N^{-j}\right) ^{-s}\left( v_{n}-v\right) \ast
\varphi _{j}\right\vert _{L_{p}\left( \mathbf{R}^{d};V_{r}\right)
}^{p}\right) ^{1/p}=0.
\end{equation*}
\end{proof}

An obvious consequence of Lemma \ref{lem-approximation1} (the form of the
approximating sequence is identical for different $V$) is the following

\begin{lemma}
\label{lem-denseSubspace} Let $p\geq 1$ and $s,s^{\prime }\in \mathbf{R}$.
Then the set $\tilde{\mathcal{\mathbb{C}}}_{p,p}^{\infty }\left( E\right) $
is a dense subset in $\mathbb{B}_{p,pp}^{s^{\prime }}\left( E\right) ,\tilde{%
\mathcal{\mathbb{C}}}_{0,p}^{\infty }\left( \mathbf{R}^{d}\right) $ is a
dense subset of $\mathbb{B}_{pp}^{s^{\prime }}\left( \mathbf{R}^{d}\right) ,$
and $\tilde{\mathcal{\mathbb{C}}}_{r,p}^{\infty }\left( E\right) $ is dense
in $\mathbb{H}_{r,p}^{s}\left( E\right) ,r=0,2.$ Moreover, the set $\tilde{%
\mathcal{\mathbb{C}}}_{2,p}^{\infty }\left( E\right) \cap \tilde{\mathcal{%
\mathbb{C}}}_{p,p}^{\infty }\left( E\right) $ is a dense subset of $\mathbb{B%
}_{p,pp}^{s^{\prime }}\left( E\right) \cap \mathbb{H}_{2,p}^{s}\left(
E\right) $.
\end{lemma}

\subsection{Representation of fractional operator and some density estimates}

We will use repeatedly the following representation of the fractional
operator. Let $\mu \in \mathfrak{A}_{sym}^{\sigma }=\left\{ \eta \in 
\mathfrak{A}^{\sigma }:\eta \text{ is symmetric, }\eta =\eta _{sym}\right\}
. $ Then for $\delta \in (0,1)$ and $f\in \mathcal{S}\left( \mathbf{R}%
^{d}\right) $, we have%
\begin{eqnarray*}
&&-\left( -\psi ^{\mu }\left( \xi \right) \right) ^{\delta }\hat{f}\left(
\xi \right) \\
&=&c_{\delta }\int_{0}^{\infty }t^{-\delta }\left[ \exp \left( \psi ^{\mu
}\left( \xi \right) t\right) -1\right] \frac{dt}{t}\hat{f}\left( \xi \right)
,\xi \in \mathbf{R}^{d},
\end{eqnarray*}%
and%
\begin{eqnarray}
L^{\mu ;\delta }f\left( x\right) &=&\mathcal{F}^{-1}\left[ -\left( -\psi
^{\mu }\right) ^{\delta }\hat{f}\right] \left( x\right)  \label{ff01} \\
&=&c_{\delta }\mathbf{E}\int_{0}^{\infty }t^{-\delta }\left[ f\left(
x+Z_{t}^{\mu }\right) -f\left( x\right) \right] \frac{dt}{t},x\in \mathbf{R}%
^{d}.  \notag
\end{eqnarray}

\begin{lemma}
\label{le0}Let $\mu \in \mathfrak{A}_{sym}^{\sigma },\delta \in \left(
0,1\right) .$

a) For any $p\geq 1,$ and $\varepsilon >0$ there is $C$ so that%
\begin{equation*}
\left\vert L^{\mu ;\delta }f\right\vert _{L_{p}\left( \mathbf{R}^{d}\right)
}\leq \varepsilon \left\vert L^{\mu }f\right\vert _{L_{p}\left( \mathbf{R}%
^{d}\right) }+C\left\vert f\right\vert _{L_{p}\left( \mathbf{R}^{d}\right)
},f\in \mathcal{S}\left( \mathbf{R}^{d}\right) .
\end{equation*}

b) Let \textbf{D}$\left( \kappa ,l\right) $ and \textbf{B}$\left( \kappa
,l\right) $ hold for $\pi \in \mathfrak{A}^{\sigma }$ with scaling function $%
\kappa $ and scaling factor $l$, and 
\begin{equation*}
\int_{\left\vert y\right\vert \leq 1}\left\vert y\right\vert ^{\alpha _{1}}d%
\widetilde{\mu }_{R}+\int_{\left\vert y\right\vert >1}\left\vert
y\right\vert ^{\alpha _{2}}d\widetilde{\mu }_{R}\leq M,R>0
\end{equation*}%
($\alpha _{1},\alpha _{2}$ are exponents in \textbf{B}$\left( \kappa
,l\right) $). Let $p^{R}\left( t,x\right) =p^{\tilde{\pi}_{R}}\left(
t,x\right) ,x\in \mathbf{R}^{d}$, be the pdf of $Z_{t}^{\tilde{\pi}%
_{R}},t>0,R>0.$ Then for each $\beta \in \lbrack 0,\delta \alpha _{2})$
there is $C=C\left( \kappa ,l,N_{0},\beta \right) $ ($N_{0}$ is a constant
in \textbf{B}$\left( \kappa ,l\right) $) so that for $\left\vert
k\right\vert \leq 2,$ 
\begin{eqnarray*}
\int \left( 1+\left\vert x\right\vert ^{\beta }\right) \left\vert D^{k}L^{%
\tilde{\mu}_{R};\delta }p^{R}\left( 1,x\right) \right\vert dx &\leq &CM, \\
\int (1+\left\vert x\right\vert ^{\beta })\left\vert L^{\tilde{\mu}%
_{R};\delta }L^{\tilde{\pi}^{\ast }}p^{R}\left( 1,x\right) \right\vert dx
&\leq &CM.
\end{eqnarray*}
\end{lemma}

\begin{proof}
Indeed for any $a>0,f\in \mathcal{S}\left( \mathbf{R}^{d}\right) ,x\in 
\mathbf{R}^{d},$ by Ito formula and (\ref{ff01}),%
\begin{eqnarray}
L^{\mu ;\delta }f\left( x\right) &=&c\mathbf{E}\int_{0}^{a}t^{-\delta
}\int_{0}^{t}L^{\mu }f\left( x+Z_{r}^{\mu }\right) dr\frac{dt}{t}
\label{f00} \\
&&+c\mathbf{E}\int_{a}^{\infty }t^{-\delta }\left[ f\left( x+Z_{t}^{\mu
}\right) -f\left( x\right) \right] \frac{dt}{t}.  \notag
\end{eqnarray}%
The statement a) follows by Minkowski inequality.

By (\ref{f00}), for $\left\vert k\right\vert \leq 2,$ 
\begin{eqnarray*}
&&\int \left( 1+\left\vert x\right\vert ^{\beta }\right) \left\vert L^{\mu
;\delta }D^{k}p^{R}\left( 1,x\right) \right\vert dx \\
&\leq &C\mathbf{E}\int_{0}^{1}t^{-\delta }\int_{0}^{t}\int \left(
1+\left\vert x\right\vert ^{\beta }\right) \left\vert L^{\mu
}D^{k}p^{R}\left( 1,x+Z_{r}^{\mu }\right) \right\vert dxdr\frac{dt}{t} \\
&&+C\mathbf{E}\int_{1}^{\infty }t^{-\delta }\int \left( 1+\left\vert
x\right\vert ^{\beta }\right) [\left\vert D^{k}p^{R}\left( 1,x+Z_{t}^{\mu
}\right) \right\vert +\left\vert D^{k}p^{R}\left( 1,x\right) \right\vert ]dx%
\frac{dt}{t} \\
&=&A_{1}+A_{2}.
\end{eqnarray*}%
Now, by Corollary 4 and Lemma 17 in \cite{Mph2}, 
\begin{eqnarray*}
A_{1} &\leq &C\int_{0}^{1}t^{-\delta }\int \left( 1+\left\vert x\right\vert
^{\beta }\right) \left\vert L^{\mu }D^{k}p^{R}\left( 1,x\right) \right\vert
dxdt \\
&&+C\int_{0}^{1}t^{-\delta }\int_{0}^{t}\mathbf{E}\left( \left\vert
Z_{r}^{\mu }\right\vert ^{\beta }\right) dr\int \left\vert L^{\mu
}D^{k}p^{R}\left( 1,x\right) \right\vert dx\frac{dt}{t} \\
&\leq &C,
\end{eqnarray*}%
and%
\begin{eqnarray*}
A_{2} &\leq &C\int_{1}^{\infty }t^{-\delta }\int \left( 1+\left\vert
x\right\vert ^{\beta }\right) \left\vert D^{k}p^{R}\left( 1,x\right)
\right\vert ]dx\frac{dt}{t} \\
&&+C\int_{1}^{\infty }t^{-\delta }\mathbf{E}\left( \left\vert Z_{t}^{\mu
}\right\vert ^{\beta }\right) \int \left\vert D^{k}p^{R}\left( 1,x\right)
\right\vert dx\frac{dt}{t} \\
&\leq &C\left( 1+\int_{1}^{\infty }t^{-\delta }t^{\frac{\beta }{\alpha _{2}}}%
\frac{dt}{t}\right) \leq C.
\end{eqnarray*}%
Similarly, the second inequality of part b) is proved.
\end{proof}

Let $\mathfrak{A}_{sign}^{\sigma }=\mathfrak{A}^{\sigma }-\mathfrak{A}%
^{\sigma }=\left\{ \eta -\rho :\eta ,\rho \in \mathfrak{A}^{\sigma }\right\}
.$

\begin{lemma}
\label{al1}Let $\delta \in \left( 0,1\right) ,$\textbf{D}$\left( \kappa
,l\right) $ and \textbf{B}$\left( \kappa ,l\right) $ hold for $\pi \in 
\mathfrak{A}^{\sigma }$ with scaling function $\kappa $ and scaling factor $%
l $. Let $\mu \in \mathfrak{A}_{sym}^{\sigma }$ and $\eta \in \mathfrak{A}%
_{sign}^{\sigma }.$ Then%
\begin{eqnarray*}
p^{\pi }\left( t,x\right) &=&a\left( t\right) ^{-d}p^{\tilde{\pi}_{a\left(
t\right) }}\left( 1,xa\left( t\right) ^{-1}\right) ,x\in \mathbf{R}^{d},t>0,
\\
L^{\mu ;\delta }p^{\pi }\left( t,x\right) &=&\frac{1}{t^{\delta }}a\left(
t\right) ^{-d}(L^{\tilde{\mu}_{a\left( t\right) };\delta }p^{\tilde{\pi}%
_{a\left( t\right) }})\left( 1,xa\left( t\right) ^{-1}\right) ,x\in \mathbf{R%
}^{d},t>0, \\
L^{\eta }L^{\mu ;\delta }p^{\pi }\left( t,x\right) &=&\frac{1}{t^{1+\delta }}%
a\left( t\right) ^{-d}(L^{\tilde{\eta}_{a\left( t\right) }}L^{\tilde{\mu}%
_{a\left( t\right) };\delta }p^{\tilde{\pi}_{a\left( t\right) }})\left(
1,xa\left( t\right) ^{-1}\right) ,x\in \mathbf{R}^{d},t>0,
\end{eqnarray*}%
where $a\left( t\right) =\inf \left\{ r\geq 0:\kappa \left( r\right) \geq
t\right\} ,t>0.$
\end{lemma}

\begin{proof}
Indeed, by Lemma 5 in \cite{MPh}, for each $t>0$ and $r>0$, the density $p^{%
\tilde{\pi}_{a\left( t\right) }}\left( r,x\right) ,x\in \mathbf{R}^{d},$ is
4 times continuously differentiable in $x$ bounded and integrable. Obviously,%
\begin{equation*}
\exp \left\{ \psi ^{\pi }\left( \xi \right) t\right\} =\exp \left\{ \psi ^{%
\tilde{\pi}_{a\left( t\right) }}\left( a\left( t\right) \xi \right) \right\}
,t>0,\xi \in \mathbf{R}^{d},
\end{equation*}%
and%
\begin{eqnarray*}
&&\left( -\psi ^{\mu }\left( \xi \right) \right) ^{\delta }\exp \left\{ \psi
^{\pi }\left( \xi \right) t\right\} \\
&=&\frac{1}{t^{\delta }}\left( -\psi ^{\tilde{\mu}_{a\left( t\right)
}}\left( a\left( t\right) \xi \right) \right) ^{\delta }\exp \left\{ \psi ^{%
\tilde{\pi}_{a\left( t\right) }}\left( a\left( t\right) \xi \right) \right\}
,t>0,\xi \in \mathbf{R}^{d}.
\end{eqnarray*}%
We derive the first two equalities by taking Fourier inverse. Similarly, the
third equality can be derived. The claim follows.
\end{proof}

\begin{lemma}
\label{al2}Let $\delta \in \left( 0,1\right) ,$\textbf{D}$\left( \kappa
,l\right) $ and \textbf{B}$\left( \kappa ,l\right) $ hold for $\pi \in 
\mathfrak{A}^{\sigma }$ with scaling function $\kappa $ and scaling factor $%
l $. Let $\mu \in \mathfrak{A}_{sym}^{\sigma }$. Assume 
\begin{equation*}
\int_{\left\vert y\right\vert \leq 1}\left\vert y\right\vert ^{\alpha _{1}}d%
\widetilde{\mu }_{R}+\int_{\left\vert y\right\vert >1}\left\vert
y\right\vert ^{\alpha _{2}}d\widetilde{\mu }_{R}\leq M,R>0
\end{equation*}%
($\alpha _{1},\alpha _{2}$ are exponents in \textbf{B}$\left( \kappa
,l\right) $). Then there exists $C=C\left( \kappa ,l,N_{0}\right) >0$ such
that for $\left\vert k\right\vert \leq 2,\beta \in \lbrack 0,\delta \alpha
_{2}),$\ 
\begin{eqnarray*}
\int_{\left\vert x\right\vert >c}\left\vert L^{\mu ;\delta }D^{k}p^{\pi
}\left( t,x\right) \right\vert dx &\leq &CMt^{-\delta }a\left( t\right)
^{\beta -\left\vert k\right\vert }c^{-\beta }, \\
\int \left\vert L^{\mu ;\delta }D^{k}p^{\pi }\left( t,x\right) \right\vert
dx &\leq &CMt^{-\delta }a\left( t\right) ^{-\left\vert k\right\vert },
\end{eqnarray*}%
with $a\left( t\right) =\inf \left\{ r\geq 0:\kappa \left( r\right) \geq
t\right\} ,t>0.$ Recall $\alpha _{1},\alpha _{2}\in (0,1]$ if $\sigma \in
\left( 0,1\right) ;\alpha _{1},\alpha _{2}\in (1,2]$ if $\sigma \in \left(
1,2\right) $ and $\alpha _{2}\in (0,1),\alpha _{1}\in (1,2]$ if $\sigma =1.$
\end{lemma}

\begin{proof}
Indeed, by Lemma \ref{al1}, Chebyshev inequality, and Lemma \ref{le0}, for $%
\left\vert k\right\vert \leq 2,\beta \in \lbrack 0,\delta \alpha _{2}),$%
\begin{eqnarray*}
&&\int_{\left\vert x\right\vert >c}\left\vert L^{\mu ;\delta }D^{k}p^{\pi
}\left( t,x\right) \right\vert dx \\
&=&\frac{1}{t^{\delta }}a\left( t\right) ^{-d-k}\int_{\left\vert
x\right\vert >c}\left\vert L^{\tilde{\mu}_{a\left( t\right) };\delta
}D^{k}p^{\tilde{\pi}_{a\left( t\right) }}\left( 1,\frac{x}{a\left( t\right) }%
\right) \right\vert dx \\
&\leq &\frac{a\left( t\right) ^{\beta -k}c^{-\beta }}{t^{\delta }}\int
\left\vert x\right\vert ^{\beta }\left\vert L^{\tilde{\mu}_{a\left( t\right)
};\delta }D^{k}p^{\tilde{\pi}_{a\left( t\right) }}\left( 1,x\right)
\right\vert dx\leq CM\frac{a\left( t\right) ^{\beta -k}c^{-\beta }}{%
t^{\delta }}.
\end{eqnarray*}%
Similarly, we derive the second estimate.
\end{proof}

\begin{lemma}
\label{mvt}Let \textbf{D}$\left( \kappa ,l\right) $ and \textbf{B}$\left(
\kappa ,l\right) $ hold for $\pi \in \mathfrak{A}^{\sigma }$ with scaling
function $\kappa $ and scaling factor $l$. Let $\mu \in \mathfrak{A}%
_{sym}^{\sigma }$. Assume 
\begin{equation*}
\int_{\left\vert y\right\vert \leq 1}\left\vert y\right\vert ^{\alpha _{1}}d%
\widetilde{\mu }_{R}+\int_{\left\vert y\right\vert >1}\left\vert
y\right\vert ^{\alpha _{2}}d\tilde{\mu}_{R}\leq M,R>0
\end{equation*}%
($\alpha _{1},\alpha _{2}$ are exponents in \textbf{B}$\left( \kappa
,l\right) $). Then for $\delta \in \left( 0,1\right) ,$

a) There exists $C=C\left( \kappa ,l,N_{0}\right) >0$ such that%
\begin{equation*}
\int_{\mathbf{R}^{d}}\left\vert L^{\mu ;\delta }p^{\pi }\left( t,x-y\right)
-L^{\mu ;\delta }p^{\pi }\left( t,x\right) \right\vert dx\leq CM\frac{%
\left\vert y\right\vert }{t^{\delta }a\left( t\right) },t>0,y\in \mathbf{R}%
^{d},
\end{equation*}%
where $a\left( t\right) =\inf \left\{ r:\kappa \left( r\right) \geq
t\right\} ,t>0.$

b) There is a constant $C=C\left( \kappa ,l,N_{0}\right) $ such that

\begin{align}
& \int_{2a}^{\infty }\left( \int \left\vert L^{\mu ;\frac{1}{2}}p^{\pi
}\left( t-s,x\right) -L^{\mu ;\frac{1}{2}}p^{\pi }\left( t,x\right)
\right\vert dx\right) ^{2}dt  \label{eq:MVTtime} \\
& \leq CM,\left\vert s\right\vert \leq a<\infty .  \notag
\end{align}
\end{lemma}

\begin{proof}
By Lemma \ref{al1} and Lemma \ref{le0},%
\begin{eqnarray*}
&&\int_{\mathbf{R}^{d}}\left\vert L^{\mu ;\delta }p^{\pi }\left(
t,x-y\right) -L^{\mu ;\delta }p^{\pi }\left( t,x\right) \right\vert dx \\
&=&\frac{1}{t^{\delta }}\int \left\vert L^{\mu _{a\left( t\right) };\delta
}p^{\tilde{\pi}_{a\left( t\right) }}\left( 1,x-\frac{y}{a\left( t\right) }%
\right) -L^{\mu _{a\left( t\right) };\delta }p^{\tilde{\pi}_{a\left(
t\right) }}\left( 1,x\right) \right\vert dx \\
&\leq &\frac{1}{t^{\delta }}\int_{0}^{1}\int \left\vert \nabla L^{\tilde{\mu}%
_{a\left( t\right) };\delta }p^{\tilde{\pi}_{a\left( t\right) }}\left( 1,x-s%
\frac{y}{a\left( t\right) }\right) \right\vert \frac{\left\vert y\right\vert 
}{a\left( t\right) }dxds \\
&\leq &C\frac{\left\vert y\right\vert }{t^{\delta }a\left( t\right) }\int
\left\vert L^{\tilde{\mu}_{a\left( t\right) };\delta }\nabla p^{\tilde{\mu}%
_{a(t)}}\left( 1,x\right) \right\vert dx\leq CM\frac{\left\vert y\right\vert 
}{t^{\delta }a\left( t\right) }.
\end{eqnarray*}%
Similarly, we derive the estimate (\ref{eq:MVTtime}). By Lemma \ref{al1} and
Lemma \ref{le0},%
\begin{eqnarray*}
&&\int_{2a}^{\infty }\left( \int \left\vert L^{\mu ;\frac{1}{2}}p^{\pi
}\left( t-s,x\right) -L^{\mu ;\frac{1}{2}}p^{\pi }\left( t,x\right)
\right\vert dx\right) ^{2}dt \\
&\leq &\left\vert s\right\vert ^{2}\int_{2a}^{\infty }\left(
\int_{0}^{1}\int \left\vert L^{\mu ;\frac{1}{2}}L^{\pi }p^{\pi }\left(
t-rs,x\right) \right\vert dxdr\right) ^{2}dt \\
&\leq &C\left\vert s\right\vert ^{2}\int_{2a}^{\infty }\left( \int_{0}^{1}%
\frac{dr}{\left( t-rs\right) ^{1+\frac{1}{2}}}\right) ^{2}dt\leq
C\int_{2a}^{\infty }\left\vert \frac{1}{\left( t-s\right) ^{\frac{1}{2}}}-%
\frac{1}{t^{\frac{1}{2}}}\right\vert ^{2}dt \\
&\leq &C\left\vert s\right\vert \int_{2a}^{\infty }\frac{dt}{\left(
t-s\right) t}=C\int_{2a}^{\infty }\frac{1}{\left( \frac{t}{s}-1\right) }%
\frac{dt}{t}\leq C.
\end{eqnarray*}
\end{proof}

\section{Proof of the main theorem}

We split the proof into several steps. First we derive the existence of
smooth solutions for the equation with smooth input functions. Then we prove
the main estimate for them by verifying H\"{o}rmander condition. At the end
we extend the estimates and regularity result for general input functions.

\subsection{Existence and uniqueness of solution for smooth input functions}

Let $\pi \in \mathfrak{A}^{\sigma }$, and $Z_{t}=Z_{t}^{\pi },t\geq 0,$ be
the Levy process associated to it. Let $P_{t}\left( dy\right) $ be the
distribution of $Z_{t}^{\pi },t>0$, and for a measurable $f\geq 0,$%
\begin{equation*}
T_{t}f\left( x\right) =\int f\left( x+y\right) P_{t}\left( dy\right) ,\left(
t,x\right) \in E.
\end{equation*}%
For the representation of the solution to (\ref{mainEq}) we will use the
following operators:

\begin{align*}
T_{t}^{\lambda }g\left( x\right) & =e^{-\lambda t}\int g\left( x+y\right)
P_{t}\left( dy\right) ,\left( t,x\right) \in E,\hspace{1em}g\in \tilde{%
\mathcal{\mathbb{C}}}_{0,p}^{\infty }\left( \mathbf{R}^{d}\right) ,p>1, \\
R_{\lambda }f\left( t,x\right) & =\int_{0}^{t}e^{-\lambda \left( t-s\right)
}\int f\left( s,x+y\right) P_{t-s}\left( dy\right) ds,\left( t,x\right) \in
E,\hspace{1em}f\in \tilde{\mathcal{\mathbb{C}}}_{0,p}^{\infty }\left(
E\right) ,p>1,
\end{align*}%
and%
\begin{eqnarray*}
\tilde{R}_{\lambda }\Phi \left( t,x\right) &=&\int_{0}^{t}e^{-\lambda \left(
t-s\right) }\int_{U}\int \Phi \left( s,x+y,z\right) P_{t-s}\left( dy\right)
q\left( ds,dz\right) ,\left( t,x\right) \in E,\hspace{1em} \\
\Phi &\in &\tilde{\mathcal{\mathbb{C}}}_{2,p}^{\infty }\left( E\right) \cap 
\tilde{\mathcal{\mathbb{C}}}_{p,p}^{\infty }\left( E\right) \text{ if }p\geq
2, \\
\Phi &\in &\tilde{\mathcal{\mathbb{C}}}_{p,p}^{\infty }\left( E\right) \text{
if }p\in \left( 1,2\right) .
\end{eqnarray*}

First we present some simple estimates of $T_{t}^{\lambda }g,R_{\lambda }f,%
\tilde{R}_{\lambda }\Phi $.

\begin{lemma}
\label{lem-basicLpEst}The following estimates hold for any multiindex $%
\gamma $:

(i) $\mathbf{P}$-a.s.%
\begin{eqnarray*}
\left\vert D^{\gamma }T^{\lambda }g\right\vert _{L_{p}\left( E\right) }
&\leq &\rho _{\lambda }^{\frac{1}{p}}\left\vert D^{\gamma }g\right\vert
_{L_{p}\left( \mathbf{R}^{d}\right) },\hspace{1em}g\in \tilde{\mathbb{C}}%
_{0,p}^{\infty }\left( \mathbf{R}^{d}\right) ,p\geq 1, \\
\left\vert D^{\gamma }R_{\lambda }f\right\vert _{L_{p}\left( E\right) }
&\leq &\rho _{\lambda }\left\vert D^{\gamma }f\right\vert _{L_{p}\left(
E\right) },\hspace{1em}f\in \mathbb{\tilde{C}}_{0,p}^{\infty }\left(
E\right) ,p\geq 1,
\end{eqnarray*}%
and 
\begin{eqnarray*}
\left\vert D^{\gamma }R_{\lambda }f\left( t,\cdot \right) \right\vert
_{L_{p}\left( \mathbf{R}^{d}\right) } &\leq &\int_{0}^{t}\left\vert
D^{\gamma }f\left( s,\cdot \right) \right\vert _{L_{p}\left( \mathbf{R}%
^{d}\right) }ds,t\geq 0, \\
\left\vert T_{t}^{\lambda }g\right\vert _{L_{p}\left( \mathbf{R}^{d}\right)
} &\leq &e^{-\lambda t}\left\vert g\right\vert _{L_{p}\left( \mathbf{R}%
^{d}\right) },t\geq 0,p\geq 1;
\end{eqnarray*}

(ii) For each $p\geq 2,$%
\begin{eqnarray*}
&&\left\vert D^{\gamma }\tilde{R}_{\lambda }\Phi \right\vert _{\mathbb{L}%
_{p}\left( E\right) }^{p} \\
&\leq &C\left[ \rho _{\lambda }^{\frac{p}{2}}\mathbf{E}\int_{0}^{T}\left%
\vert D^{\gamma }\Phi \left( s,\cdot \right) \right\vert _{L_{2,p}\left( 
\mathbf{R}^{d}\right) }^{p}ds+\rho _{\lambda }\left\vert D^{\gamma }\Phi
\right\vert _{\mathbb{L}_{p,p}\left( E\right) }^{p}\right] , \\
\Phi  &\in &\tilde{\mathbb{C}}_{2,p}^{\infty }\left( E\right) \cap \tilde{%
\mathbb{C}}_{p,p}^{\infty }\left( E\right) ,
\end{eqnarray*}%
and for each \thinspace $p\in (1,2),$%
\begin{equation*}
\left\vert D^{\gamma }\tilde{R}_{\lambda }\Phi \right\vert _{\mathbb{L}%
_{p}\left( E\right) }^{p}\leq C\rho _{\lambda }\left\vert D^{\gamma }\Phi
\right\vert _{\mathbb{L}_{p,p}\left( E\right) }^{p},\Phi \in \tilde{\mathbb{C%
}}_{p,p}^{\infty }\left( E\right) ,
\end{equation*}

where $\rho _{\lambda }=T\wedge \frac{1}{\lambda }.$ Moreover,%
\begin{eqnarray*}
&&\left\vert D^{\gamma }\tilde{R}_{\lambda }\Phi \left( t,\cdot \right)
\right\vert _{\mathbb{L}_{p}\left( \mathbf{R}^{d}\right) }^{p} \\
&\leq &C\left\{ \mathbf{E}\left[ \left( \int_{0}^{t}\left\vert D^{\gamma
}\Phi \left( s,\cdot \right) \right\vert _{L_{2,p}\left( \mathbf{R}%
^{d}\right) }^{2}ds\right) ^{p/2}\right] +\mathbf{E}\int_{0}^{t}\left\vert
D^{\gamma }\Phi \left( s,\cdot \right) \right\vert _{L_{p,p}\left( \mathbf{R}%
^{d}\right) }^{p}ds\right\} ,
\end{eqnarray*}%
if $p\geq 2$, and%
\begin{equation*}
\left\vert D^{\gamma }\tilde{R}_{\lambda }\Phi \left( t,\cdot \right)
\right\vert _{\mathbb{L}_{p}\left( \mathbf{R}^{d}\right) }^{p}\leq C\mathbf{E%
}\int_{0}^{t}\left\vert D^{\gamma }\Phi \left( s,\cdot \right) \right\vert
_{L_{p,p}\left( \mathbf{R}^{d}\right) }^{p}ds,t>0,
\end{equation*}%
if $p\in \left( 1,2\right) .$
\end{lemma}

\begin{proof}
The estimates (i) follow from Lemma 15 in \cite{Mph2} and Lemma 8 in \cite%
{MPh}. Let $p\geq 2,\Phi \in \tilde{\mathbb{C}}_{2,p}^{\infty }\left(
E\right) \cap \tilde{\mathbb{C}}_{p,p}^{\infty }\left( E\right) $. Recall $%
\Phi =\Phi \chi _{U_{n}}$ for some $U_{n}\in \mathcal{U}$ with $\pi \left(
U_{n}\right) <\infty .$ Obviously, for any multiindex $\gamma $, $\left(
t,x\right) \in E,$%
\begin{eqnarray*}
D^{\gamma }\tilde{R}_{\lambda }\Phi \left( t,x\right)
&=&\int_{0}^{t}e^{-\lambda \left( t-s\right) }\int_{U}\int D^{\gamma }\Phi
\left( s,x+y,z\right) P_{t-s}\left( dy\right) q\left( ds,dz\right) \\
&=&\int_{0}^{t}e^{-\lambda \left( t-s\right) }\int_{U}\int D^{\gamma }\Phi
\left( s,x+y,z\right) P_{t-s}\left( dy\right) p\left( ds,dz\right) \\
&&-\int_{0}^{t}e^{-\lambda \left( t-s\right) }\int_{U}\int D^{\gamma }\Phi
\left( s,x+y,z\right) P_{t-s}\left( dy\right) \pi \left( dz\right) ds.
\end{eqnarray*}%
By Kunita's inequality (see \cite{ku}, \cite{mp2}), for $t>0,$ 
\begin{eqnarray*}
&&\mathbf{E}\int \left\vert D^{\gamma }\tilde{R}_{\lambda }\Phi \left(
t,x\right) \right\vert ^{p}dx \\
&\leq &C\mathbf{E}\int \left( \int_{0}^{t}e^{-2\lambda \left( t-s\right)
}\int \left( \int_{U}\left\vert D^{\gamma }\Phi \left( s,x+y,z\right)
\right\vert ^{2}\pi \left( dz\right) \right) P_{t-s}\left( dy\right)
ds\right) ^{p/2}dx \\
&&+C\mathbf{E}\int \int_{0}^{t}e^{-p\lambda \left( t-s\right) }\left\vert
T_{t-s}D^{\gamma }\Phi \left( s,x,z\right) \right\vert ^{p}\pi \left(
dz\right) dsdx \\
&=&B\left( t\right) +D\left( t\right) .
\end{eqnarray*}

By Fubini theorem and Minkowski inequality,%
\begin{equation*}
D\left( t\right) \leq C\mathbf{E}\int_{0}^{t}e^{-p\lambda \left( t-s\right)
}\left\vert D^{\gamma }\Phi \left( s,\cdot \right) \right\vert
_{L_{p,p}\left( \mathbf{R}^{d}\right) }^{p}ds,t>0,
\end{equation*}

and%
\begin{eqnarray*}
B\left( t\right) &\leq &C\mathbf{E}\left[ \left( \int_{0}^{t}e^{-2\lambda
\left( t-s\right) }\left\vert D^{\gamma }\Phi \left( s,\cdot \right)
\right\vert _{L_{2,p}\left( \mathbf{R}^{d}\right) }^{2}ds\right) ^{p/2}%
\right] \\
&=&C\mathbf{E}\left[ \left( \int_{0}^{t}e^{-2\lambda \left( t-s\right)
}\left\vert D^{\gamma }\Phi \left( s,\cdot \right) \right\vert
_{L_{2,p}\left( \mathbf{R}^{d}\right) }^{2}ds\right) ^{p/2}\right] \\
&\leq &C\left( \frac{1}{\lambda }\right) ^{p/2}\mathbf{E}\left[
\int_{0}^{t}2\lambda e^{-2\lambda \left( t-s\right) }\left\vert D^{\gamma
}\Phi \left( s,\cdot \right) \right\vert _{L_{2,p}\left( \mathbf{R}%
^{d}\right) }^{p}ds\right] .
\end{eqnarray*}%
Now,%
\begin{equation*}
\int_{0}^{T}D\left( t\right) dt\leq C\rho _{\lambda }\mathbf{E}%
\int_{0}^{T}\left\vert D^{\gamma }\Phi \left( s,\cdot \right) \right\vert
_{L_{p,p}\left( \mathbf{R}^{d}\right) }^{p}ds
\end{equation*}%
and%
\begin{equation*}
\int_{0}^{T}B\left( t\right) dt\leq C\rho _{\lambda }^{\frac{p}{2}}\mathbf{E}%
\int_{0}^{T}\left\vert D^{\gamma }\Phi \left( s,\cdot \right) \right\vert
_{L_{2,p}\left( \mathbf{R}^{d}\right) }^{p}ds.
\end{equation*}

Similarly we consider the case $p\in (1,2).$
\end{proof}

\begin{lemma}
\label{lem:basicEst2}For $\mu {\in }\mathfrak{A}$, let $f\in \mathcal{%
\mathbb{\tilde{C}}}_{0,p}^{\infty }\left( E\right) ,g\in \tilde{\mathcal{%
\mathbb{C}}}_{0,p}^{\infty }\left( \mathbf{R}^{d}\right) ,\Phi \in \mathbb{%
\tilde{\mathcal{\mathbb{C}}}}_{2,p}^{\infty }\left( E\right) \cap {\tilde{%
\mathcal{\mathbb{C}}}}_{p,p}^{\infty }\left( E\right) $ for $p\in \left[
2,\infty \right) $ and $\Phi \in {\tilde{\mathcal{\mathbb{C}}}}%
_{p,p}^{\infty }\left( E\right) $ for $p\in (1,2$, then there is unique $%
u\in \tilde{\mathcal{\mathbb{C}}}_{0,p}^{\infty }\left( E\right) $ solving (%
\ref{mainEq}). Moreover,

\begin{equation*}
u\left( t,x\right) =T_{t}^{\lambda }g\left( x\right) +R_{\lambda }f\left(
t,x\right) +\tilde{R}_{\lambda }\Phi \left( t,x\right) ,\left( t,x\right)
\in E,
\end{equation*}%
and $u_{1}\left( t,x\right) =T_{t}^{\lambda }g\left( x\right) ,\left(
t,x\right) \in E,$ solves (\ref{mainEq}) with $f=0,\Phi =0$, $%
u_{2}=R_{\lambda }f$ solves (\ref{mainEq}) with $g=0,\Phi =0$, and $u_{3}=%
\tilde{R}_{\lambda }\Phi $ solves \ (\ref{mainEq})with $g=0,f=0.$
\end{lemma}

\begin{proof}
Uniqueness is a simple repeat of the proof of Lemma 8 in \cite{MPh}. We
prove that $u_{1},u_{2}$ solve the corresponding equations by repeating the
proofs of Lemma 8 in \cite{MPh} and Lemma 15\ in \cite{Mph2}. Let $\Phi \in 
\mathbb{\tilde{\mathcal{\mathbb{C}}}}_{2,p}^{\infty }\left( E\right) \cap {%
\tilde{\mathcal{\mathbb{C}}}}_{p,p}^{\infty }\left( E\right) $ if $p\in %
\left[ 2,\infty \right) $ or $\Phi \in {\tilde{\mathcal{\mathbb{C}}}}%
_{p,p}^{\infty }\left( E\right) $ if $p\in (1,2]$. Recall that $\Phi =\Phi
\chi _{U_{n}}$ for some $U_{n}\in \mathcal{U}$ with $\pi \left( U_{n}\right)
<\infty $. Let $v=u_{3}=\tilde{R}_{\lambda }\Phi $. A simple application of
Ito formula and Fubini theorem show that $\mathbf{P}$-a.s.%
\begin{eqnarray*}
&&v\left( t,x\right) \\
&=&\int_{0}^{t}e^{-\lambda \left( t-s\right) }\int_{U}\int \Phi \left(
s,x+y,z\right) P_{t-s}\left( dy\right) q\left( ds,dz\right) \\
&=&\int_{0}^{t}\int_{U}\Phi \left( s,x,z\right) q\left( ds,dz\right)
+\int_{0}^{t}\int_{s}^{t}e^{-\lambda \left( r-s\right) }\times \\
&&\times \int_{U}\int [L^{\pi }\Phi \left( s,x+y,z\right) -\lambda \Phi
\left( s,x+y,z\right) ]P_{r-s}\left( dy\right) drq\left( ds,dz\right) \\
&=&\int_{0}^{t}\int_{U}\Phi \left( s,x,z\right) q\left( ds,dz\right)
+\int_{0}^{t}\left[ L^{\pi }v\left( s,x\right) -\lambda v\left( s.x\right) %
\right] ds,\left( t,x\right) \in E.
\end{eqnarray*}
\end{proof}

\section{Main estimate}

Let

\begin{equation*}
G_{s,t}^{\lambda }\left( x\right) =\exp \left( -\lambda \left( t-s\right)
\right) p^{\pi ^{\ast }}\left( t-s,x\right) ,0<s<t,x\in \mathbf{R}^{d},
\end{equation*}%
where $\pi ^{\ast }\left( dy\right) =\pi \left( -dy\right) $, and%
\begin{equation*}
u\left( t,x\right) =\int_{0}^{t}\int_{U}G_{s,t}^{\lambda }\ast \Phi \left(
s,x,\nu \right) q\left( ds,d\nu \right) ,\left( t,x\right) \in E.
\end{equation*}

The main estimate for the solution with smooth input functions is the
following statement.

\begin{lemma}
\label{lem-smoothEst2} Let $\pi ,\mu \in ${$\mathfrak{A}$}$^{\sigma }$.
Assume there is a scaling function $\kappa $ with a scaling factor $l$ such
that $D\left( \kappa ,l\right) $ and $\beta \left( \kappa ,l\right) $ hold
for both, $\pi $ and $\mu $. Let $\Phi \in \mathbb{\tilde{\mathcal{\mathbb{C}%
}}}_{2}^{\infty }\left( E\right) \cap {\tilde{\mathcal{\mathbb{C}}}}%
_{p}^{\infty }\left( E\right) $, for $p\in \left[ 2,\infty \right) $ and $%
\Phi \in {\tilde{\mathcal{\mathbb{C}}}}_{p}^{\infty }\left( E\right) $ for $%
p\in (1,2)$. Assume%
\begin{equation*}
\int_{1}^{\infty }\frac{1}{\gamma \left( t\right) ^{\beta _{0}}}\frac{dt}{t}%
<\infty
\end{equation*}%
for some $\beta _{0}<\alpha _{2}$. Then

\begin{eqnarray*}
\left\vert L^{\mu }u\right\vert _{\mathbb{L}_{p}\left( E\right) } &\leq
&C\left( \left\vert L^{\mu ;\frac{1}{2}}\Phi \right\vert _{\mathbb{L}%
_{2,p}\left( E\right) }+\left\vert \Phi \right\vert _{\mathbb{B}_{p,pp}^{1-%
\frac{1}{p}}\left( E\right) }\right) ,\hspace{1em}p\in \left[ 2,\infty
\right) \\
\left\vert L^{\mu }u\right\vert _{\mathbb{L}_{p}\left( E\right) } &\leq
&C\left\vert \Phi \right\vert _{\mathbb{B}_{p,pp}^{1-\frac{1}{p}}\left(
E\right) },\hspace{1em}p\in (1,2),
\end{eqnarray*}

where $C=C\left( \kappa ,l,p,d\right) $.
\end{lemma}

\subsubsection{Proof of Lemma \protect\ref{lem-smoothEst2}}

Let 
\begin{equation*}
G_{s,t}^{\lambda ,\varepsilon }\left( x\right) =\exp \left( -\lambda \left(
t-s\right) \right) p^{\pi ^{\ast }}\left( t-s,x\right) \chi _{\left[
\varepsilon ,\infty \right] }\left( t-s\right) ,0<s<t,x\in \mathbf{R}^{d},
\end{equation*}%
where $\pi ^{\ast }\left( dy\right) =\pi \left( -dy\right) $. Denote for $%
\varepsilon >0,$ 
\begin{eqnarray*}
Q\left( t,x\right) &=&\int_{0}^{t}\int_{U}\tilde{\Phi}_{\varepsilon }\left(
s,x,\nu \right) q\left( ds,d\nu \right) ,\left( t,x\right) \in E, \\
\tilde{\Phi}_{\varepsilon }\left( s,x,\nu \right) &=&\int \left( L^{\mu
}G_{s,t}^{\lambda ,\varepsilon }\right) \left( x-y\right) \Phi \left(
s,y,\nu \right) dy,\left( s,x\right) \in E.
\end{eqnarray*}

Obviously, with $K_{\lambda }^{\varepsilon }\left( t,x\right) =e^{-\lambda
t}L^{\pi ;\frac{1}{2}}p^{\pi ^{\ast }}\left( t,x\right) \chi _{\left[
\varepsilon ,\infty \right] }\left( t\right) ,t>0,x\in \mathbf{R}^{d},$ we
have

\begin{eqnarray}
&&\mathbf{E}\left\vert \int_{0}^{t}\int_{U}\int \left( L^{\mu
}G_{s,t}^{\lambda ,\varepsilon }\right) \left( x-y\right) \Phi \left(
s,y,\nu \right) dyq\left( ds,d\nu \right) \right\vert _{L_{p}\left( E\right)
}^{p}  \label{ff1} \\
&=&\mathbf{E}\left\vert \int_{0}^{t}\int_{U}\int L^{\mu
;1/2}G_{s,t}^{\lambda ,\varepsilon }\left( x-y\right) L^{\mu ;1/2}\Phi
\left( s,y,\nu \right) dyq\left( ds,d\nu \right) \right\vert _{L_{p}\left(
E\right) }^{p}  \notag \\
&=&\mathbf{E}\left\vert \int_{0}^{t}\int_{U}\int K_{\lambda }^{\varepsilon
}\left( t-s,x-y\right) L^{\mu ;1/2}\Phi \left( s,y,\nu \right) dyq\left(
ds,d\nu \right) \right\vert _{L_{p}\left( E\right) }^{p}.  \notag
\end{eqnarray}

If $2\leq p<\infty $, then

\begin{eqnarray*}
&&\mathbf{E}\int_{0}^{T}\left\vert Q\left( t,\cdot \right) \right\vert
_{L_{p}\left( \mathbf{R}^{d}\right) }^{p}dt \\
&\leq &C\mathbf{E}\left\{ \int_{0}^{T}\left\vert \left[ \int_{0}^{t}\int_{U}%
\tilde{\Phi}_{\varepsilon }\left( s,\cdot ,\nu \right) ^{2}\Pi \left( d\nu
\right) ds\right] ^{1/2}\right\vert _{L_{p}\left( \mathbf{R}^{d}\right)
}^{p}dt\right\} \\
&+&C\mathbf{E}\left\{ \int_{0}^{T}\int_{0}^{t}\int_{U}\left\vert \tilde{\Phi}%
_{\varepsilon }\left( s,\cdot ,\nu \right) \right\vert _{L_{p}\left( \mathbf{%
R}^{d}\right) }^{p}\Pi \left( d\nu \right) dsdt\right\} =C\left( \mathbf{E}%
I_{1}+\mathbf{E}I_{2}\right) .
\end{eqnarray*}

If $1<p<2$, then by Lemma 9 (see also Remark 1 therein) in \cite{MF},

\begin{equation*}
\mathbf{E}\int_{0}^{T}\left\vert Q\left( t,\cdot \right) \right\vert
_{L_{p}\left( \mathbf{R}^{d}\right) }^{p}dt\leq C\mathbf{E}I_{2}.
\end{equation*}

\emph{Estimate of} $\mathbf{E}I_{2}$. Let $B_{t}^{\lambda }g\left( x\right)
=e^{-\lambda t}\mathbf{E}g\left( x+Z_{t}^{\pi }\right) ,\left( t,x\right)
\in E,g\in \tilde{C}_{0,p}^{\infty }\left( \mathbf{R}^{d}\right) $. Then

\begin{eqnarray*}
I_{2} &=&\int_{0}^{T}\int_{0}^{t}\int_{U}\left\vert \tilde{\Phi}%
_{\varepsilon }\left( s,\cdot ,\nu \right) \right\vert _{L_{p}\left( \mathbf{%
R}^{d}\right) }^{p}\Pi \left( d\nu \right) dsdt \\
&=&\int_{0}^{T}\int_{0}^{t}\int_{U}\left\vert \left( L^{\mu
}G_{s,t}^{\lambda ,\varepsilon }\right) \ast \Phi \left( s,\cdot ,\nu
\right) \right\vert _{L_{p}\left( \mathbf{R}^{d}\right) }^{p}\Pi \left( d\nu
\right) dsdt \\
&\leq &\int_{0}^{T}\int_{0}^{t}\int_{U}\left\vert L^{\mu }B_{t-s}^{\lambda
}\Phi \left( s,\cdot ,\nu \right) \right\vert _{L_{p}\left( \mathbf{R}%
^{d}\right) }^{p}\Pi \left( d\nu \right) dsdt \\
&=&\int_{U}\int_{0}^{T}\int_{s}^{T}\left\vert L^{\mu }B_{t-s}^{\lambda }\Phi
\left( s,\cdot ,\nu \right) \right\vert _{L_{p}\left( \mathbf{R}^{d}\right)
}^{p}dtds\Pi \left( d\nu \right)
\end{eqnarray*}

It follows from Proposition 1 (see section 4.4 as well) of \cite{Mph2} that
for $p>1,$

\begin{eqnarray*}
\mathbf{E}I_{2} &\leq &\mathbf{E}\int_{U}\int_{0}^{T}\int_{s}^{T}\left\vert
L^{\mu }B_{t-s}^{\lambda }\Phi \left( s,\cdot ,\nu \right) \right\vert
_{L_{p}\left( \mathbf{R}^{d}\right) }^{p}dtds\Pi \left( d\nu \right) \\
&\leq &C\mathbf{E}\int_{U}\int_{0}^{T}\sum_{j=0}^{\infty }\left\vert \kappa
\left( N^{-j}\right) ^{-\left( 1-1/p\right) }\left\vert \Phi \left( s,\cdot
,\nu \right) \ast \varphi _{j}\right\vert _{L_{p}\left( \mathbf{R}%
^{d}\right) }\right\vert ^{p}ds\Pi \left( d\nu \right) \\
&=&C\left\vert \Phi \right\vert _{\mathbb{B}_{p,pp}^{\mu ,N;1-1/p}\left(
E\right) }^{p}.
\end{eqnarray*}

\paragraph{Estimate of $\mathbf{E}I_{1}$}

It is enough to show that

\begin{eqnarray}
I &=&\int \int \left\{ \int \left\vert K_{\lambda }^{\varepsilon }\left(
t-s,\cdot \right) \ast \Phi \left( s,\cdot \right) \left( x\right)
\right\vert _{V_{2}}^{2}ds\right\} ^{\frac{p}{2}}dxdt  \label{ff2} \\
&\leq &C\int \int \left( \left\vert \Phi \left( t,x\right) \right\vert
_{V_{2}}^{2}\right) ^{p/2}dxdt,~\Phi \in \tilde{C}_{2,p}^{\infty }\left( 
\mathbf{R}^{d+1}\right) .  \notag
\end{eqnarray}

where $V_{2}=L_{2}\left( U,\mathcal{U},\Pi \right) $, $C$ is independent of $%
\varepsilon $ and $\Phi $.

For $p=2$, by Plancherel's theorem and Fubini's theorem, denoting $\hat{\Phi}%
=\mathcal{F}\Phi $,

\begin{eqnarray*}
&&I=\int_{\mathbf{R}^{d+1}}\int \left\vert L^{\mu ;\frac{1}{2}%
}G_{s,t}^{\lambda }\ast \Phi \left( s,x\right) \right\vert _{V_{2}}^{2}dsdxdt
\\
&=&\int_{0}^{T}\int \int_{0}^{t-\varepsilon }\int_{U}\exp \left\{ 2\left(
\psi ^{\pi }\left( \xi \right) -\lambda \right) \left( t-s\right) \right\}
\left\vert \psi ^{\mu }\left( \xi \right) \right\vert \left\vert \hat{\Phi}%
\left( s,\xi ,z\right) \right\vert ^{2}\Pi \left( dz\right) dsd\xi dt \\
&\leq &C\int_{0}^{T}\int \int_{U}\left\vert \hat{\Phi}\left( s,\xi ,z\right)
\right\vert ^{2}\Pi \left( dz\right) dsd\xi =C\int_{0}^{T}\int
\int_{U}\left\vert \Phi \left( s,x,z\right) \right\vert ^{2}\Pi \left(
dz\right) dxds
\end{eqnarray*}

Hence (\ref{ff2}) follows for $p=2$. \ 

Next we prove (\ref{ff2}) for $p>2$. According to Lemma \ref{stochHormander}
(see Appendix), it is sufficient to show that there exists $C_{0}>0$ such
that for all $\left\vert s\right\vert \leq \kappa \left( \delta \right)
,\left\vert y\right\vert \leq \delta ,\delta >0$, we have%
\begin{equation}
\mathcal{I}=\int \left[ \int \chi _{Q_{C_{0}\delta }\left( 0\right)
^{c}}\left\vert K_{\lambda }^{\varepsilon }\left( t-s,x-y\right) -K_{\lambda
}^{\varepsilon }\left( t,x\right) \right\vert dx\right] ^{2}dt\leq N,
\label{hc}
\end{equation}%
where $Q_{C_{0}\delta }\left( 0\right) =\left( -\kappa \left( C_{0}\delta
\right) ,\kappa \left( C_{0}\delta \right) \right) \times \left\{
x:\left\vert x\right\vert <C_{0}\delta \right\} .$

\subparagraph{Verification of H\"{o}rmander condition (\protect\ref{hc})}

Let

\begin{eqnarray*}
a\left( r\right) &=&\inf \left\{ t:\kappa \left( t\right) \geq r\right\}
,r>0, \\
a^{-1}\left( s\right) &=&\inf \left\{ t:a\left( t\right) \geq s\right\} ,s>0,
\\
\gamma \left( t\right) &=&\inf \left\{ r:l\left( r\right) \geq t\right\}
,t>0.
\end{eqnarray*}%
It follows from Lemma 9 in \cite{Mph2} that 
\begin{eqnarray*}
a^{-1}\left( r\right) &=&\sup_{s\leq r}\kappa \left( s\right) \leq l\left(
1\right) \kappa \left( r\right) ,r>0, \\
a^{-1}\left( r\varepsilon \right) &\leq &l\left( \varepsilon \right)
a^{-1}\left( r\right) ,\varepsilon ,r>0.
\end{eqnarray*}

and 
\begin{equation*}
a\left( \varepsilon r\right) \geq a\left( r\right) \gamma \left( \varepsilon
\right) ,r,\varepsilon >0.
\end{equation*}

In particular, $\gamma \left( \varepsilon \right) \leq a\left( \varepsilon
\right) a\left( 1\right) ^{-1},$ and 
\begin{equation}
\frac{a\left( r\right) }{a\left( r^{\prime }\right) }\leq \gamma \left( 
\frac{r^{\prime }}{r}\right) ^{-1},r^{\prime },r>0.  \label{hc1}
\end{equation}

Let $C_{0}>3$ and $3l\left( 1\right) l\left( C_{0}^{-1}\right) <1$. Now, we
follow the splitting in \cite{KK2}. Let%
\begin{equation*}
\mathcal{I}=\int_{-\infty }^{2\left\vert s\right\vert }\left[ \int ...\right]
^{2}dt+\int_{2\left\vert s\right\vert }^{\infty }\left[ \int ...\right]
^{2}dt=\mathcal{I}_{1}+\mathcal{I}_{2}.
\end{equation*}

Since $\kappa \left( C_{0}\delta \right) >3\kappa \left( \delta \right)
,\delta >0$, it follows by Lemma 9 in \cite{Mph2} (see (\ref{hc1}), and
Lemma \ref{al2}, with $k_{0}=C_{0}-1$ and $\beta \in (\frac{\beta _{0}}{2},%
\frac{\alpha _{2}}{2})$,\ 

\begin{eqnarray*}
\left\vert \mathcal{I}_{1}\right\vert &\leq &C\int_{0}^{3\left\vert
s\right\vert }\left[ \int_{\left\vert x\right\vert >k_{0}a\left( \left\vert
s\right\vert \right) }\left\vert L^{\mu ;\frac{1}{2}}p^{\pi ^{\ast }}\left(
t,x\right) \right\vert dx\right] ^{2}dt \\
&\leq &C\int_{0}^{3\left\vert s\right\vert }\left( t^{-\frac{1}{2}}a\left(
t\right) ^{\beta }\left( k_{0}a\left( \left\vert s\right\vert \right)
\right) ^{-\beta }\right) ^{2}dt\leq C\int_{0}^{3\left\vert s\right\vert
}\left( \frac{a\left( t\right) }{a\left( s\right) }\right) ^{2\beta }\frac{dt%
}{t} \\
&\leq &C\int_{0}^{3\left\vert s\right\vert }\frac{1}{\gamma \left( \frac{s}{t%
}\right) ^{2\beta }}\frac{dt}{t}\leq C\int_{1/3}^{\infty }\frac{1}{\gamma
\left( t\right) ^{2\beta }}\frac{dt}{t}
\end{eqnarray*}

Now, 
\begin{eqnarray*}
&&\left\vert \mathcal{I}_{2}\right\vert \\
&\leq &2\int_{2\left\vert s\right\vert }^{\infty }\left[ \int \chi
_{Q_{C_{0}\delta }^{c}\left( 0\right) }\left\vert L^{\mu ;\frac{1}{2}}p^{\pi
^{\ast }}\left( t-s,x-y\right) -L^{\mu ;\frac{1}{2}}p^{\pi ^{\ast }}\left(
t-s,x\right) \right\vert dx\right] ^{2}dt \\
&+&2\int_{2\left\vert s\right\vert }^{\infty }\left\{ \int \chi
_{Q_{C_{0}\delta }^{c}\left( 0\right) }|\chi _{\left[ \varepsilon ,\infty %
\right] }\left( t-s\right) L^{\mu ;\frac{1}{2}}p^{\pi ^{\ast }}\left(
t-s,x\right) -\chi _{\left[ \varepsilon ,\infty \right] }\left( t\right)
L^{\mu ;\frac{1}{2}}p^{\pi ^{\ast }}\left( t,x\right) |dx\right\} ^{2}dt \\
&=&\mathcal{I}_{2,1}+\mathcal{I}_{2,2}.
\end{eqnarray*}

We split the estimate of $\mathcal{I}_{2,1}$ into two cases.

\emph{Case 1.} Assume $\left\vert y\right\vert \leq a\left( 2\left\vert
s\right\vert \right) $. Then by Lemma \ref{mvt}a),%
\begin{eqnarray*}
&&\mathcal{I}_{2,1} \\
&\leq &C\int_{2\left\vert s\right\vert }^{\infty }\frac{\left\vert
y\right\vert ^{2}}{(t-s)a\left( t-s\right) ^{2}}dt \\
&\leq &C\left\vert y\right\vert ^{2}a\left( 2\left\vert s\right\vert \right)
^{-2}\int_{2\left\vert s\right\vert }^{\infty }\frac{a\left( 2\left\vert
s\right\vert \right) ^{2}}{a\left( t-s\right) ^{2}}\left( t-s\right) ^{-1}dt
\\
&\leq &C\int_{2\left\vert s\right\vert }^{\infty }\gamma \left( \frac{t-s}{%
2\left\vert s\right\vert }\right) ^{-2}\left( t-s\right) ^{-1}dt\leq
C\int_{1/2}^{\infty }\gamma \left( r\right) ^{-2}\frac{dr}{r}
\end{eqnarray*}

\emph{Case 2.} Assume $\left\vert y\right\vert >a\left( 2\left\vert
s\right\vert \right) $ i.e. $\delta \geq \left\vert y\right\vert >a\left(
2\left\vert s\right\vert \right) $ and $a^{-1}\left( \delta \right) \geq
a^{-1}\left( \left\vert y\right\vert \right) \geq 2\left\vert s\right\vert $%
. We split 
\begin{equation*}
\mathcal{I}_{2,1}=\int_{2\left\vert s\right\vert }^{2\left\vert s\right\vert
+a^{-1}\left( \left\vert y\right\vert \right) }\left[ \int ...\right]
^{2}+\int_{2\left\vert s\right\vert +a^{-1}\left( \left\vert y\right\vert
\right) }^{\infty }\left[ \int ...\right] ^{2}=\mathcal{I}_{2,1,1}+\mathcal{I%
}_{2,1,2}.
\end{equation*}

If $2\left\vert s\right\vert \leq t\leq 2\left\vert s\right\vert
+a^{-1}\left( \left\vert y\right\vert \right) $, then $0\leq t\leq
3a^{-1}\left( \delta \right) \leq 3l\left( 1\right) \kappa \left( \delta
\right) \leq \kappa \left( C_{0}\delta \right) $. Hence, $\left\vert
x\right\vert >C_{0}\delta \geq a\left( 2\left\vert s\right\vert \right)
+\left\vert y\right\vert $ and 
\begin{eqnarray*}
\left\vert x-y\right\vert &\geq &\left( C_{0}-1\right) \delta =k_{0}\delta
\geq \frac{k_{0}}{2}\left[ a\left( 2\left\vert s\right\vert \right)
+\left\vert y\right\vert \right] \\
&\geq &a\left( 2\left\vert s\right\vert \right) +\left\vert y\right\vert
\quad \text{if}\quad \left( t,x\right) \notin Q_{C_{0}\delta }\left(
0\right) .
\end{eqnarray*}

Also, 
\begin{equation}
2\geq \frac{2\left\vert s\right\vert +a^{-1}\left( \left\vert y\right\vert
\right) }{2\left\vert s\right\vert +a^{-1}\left( \left\vert y\right\vert
\right) -s}\geq \frac{2}{3}  \label{eq:1}
\end{equation}

and by Lemma 9 of \cite{Mph2},%
\begin{eqnarray}
&&\frac{a\left( 2\left\vert s\right\vert +a^{-1}\left( \left\vert
y\right\vert \right) \right) }{a\left( 2\left\vert s\right\vert \right)
+\left\vert y\right\vert }  \label{eq:2} \\
&\leq &\frac{a\left( 2a^{-1}\left( \left\vert y\right\vert \right) \right) }{%
a\left( 2\left\vert s\right\vert \right) +\left\vert y\right\vert }\leq
\gamma \left( 2^{-1}\right) ^{-1}\frac{a\left( a^{-1}\left( \left\vert
y\right\vert \right) \right) }{a\left( 2\left\vert s\right\vert \right)
+\left\vert y\right\vert }\leq \gamma \left( 2^{-1}\right) ^{-1}.  \notag
\end{eqnarray}

Hence, with $\beta \in (\frac{\beta _{0}}{2},\frac{\alpha _{2}}{2})$, by
Lemma 9 in \cite{Mph2} (see (\ref{hc1})), Lemma \ref{al2}, (\ref{eq:1}) and (%
\ref{eq:2}),\ 
\begin{eqnarray*}
\mathcal{I}_{2,1,1} &\leq &C\int_{2\left\vert s\right\vert }^{2\left\vert
s\right\vert +a^{-1}\left( \left\vert y\right\vert \right) }\left[
\int_{\left\vert x\right\vert >a\left( 2\left\vert s\right\vert \right)
+\left\vert y\right\vert }\left\vert L^{\mu ;\frac{1}{2}}p^{\pi ^{\ast
}}\left( t-s,x\right) \right\vert dx\right] ^{2}dt \\
&\leq &\frac{C}{\left( a\left( 2\left\vert s\right\vert \right) +\left\vert
y\right\vert \right) ^{2\beta }}\int_{2\left\vert s\right\vert
}^{2\left\vert s\right\vert +a^{-1}\left( \left\vert y\right\vert \right)
}a\left( t-s\right) ^{2\beta }\frac{dt}{\left( t-s\right) } \\
&\leq &C\frac{a\left( 2\left\vert s\right\vert +a^{-1}\left( \left\vert
y\right\vert \right) \right) ^{2\beta }}{\left[ a\left( 2\left\vert
s\right\vert \right) +\left\vert y\right\vert \right] ^{2\beta }}%
\int_{2\left\vert s\right\vert }^{2\left\vert s\right\vert +a^{-1}\left(
\left\vert y\right\vert \right) }\frac{a\left( t-s\right) ^{2\beta }}{%
a\left( 2\left\vert s\right\vert +a^{-1}\left( \left\vert y\right\vert
\right) \right) ^{2\beta }}\frac{dt}{t-s} \\
&\leq &C\int_{2\left\vert s\right\vert }^{2\left\vert s\right\vert
+a^{-1}\left( \left\vert y\right\vert \right) }\gamma \left( \frac{%
2\left\vert s\right\vert +a^{-1}\left( \left\vert y\right\vert \right) }{t-s}%
\right) ^{-2\beta }\frac{dt}{t-s} \\
&\leq &C\int_{2/3}^{\infty }\gamma \left( r\right) ^{-2\beta }\frac{dr}{r}
\end{eqnarray*}

Then by Lemma \ref{mvt}a) and (\ref{eq:2}), 
\begin{eqnarray*}
\mathcal{I}_{2,1,2} &\leq &C\int_{2\left\vert s\right\vert +a^{-1}\left(
\left\vert y\right\vert \right) }^{\infty }\left[ \int_{\mathbf{R}%
^{d}}\left\vert L^{\pi ;\frac{1}{2}}p^{\pi ^{\ast }}\left( t-s,x-y\right)
-L^{\pi ;\frac{1}{2}}p^{\pi ^{\ast }}\left( t-s,x\right) \right\vert dx%
\right] ^{2}dt \\
&\leq &C\int_{2\left\vert s\right\vert +a^{-1}\left( \left\vert y\right\vert
\right) }^{\infty }\frac{\left\vert y\right\vert ^{2}}{a\left( t-s\right)
^{2}}\frac{dt}{t-s} \\
&=&C\frac{\left\vert y\right\vert ^{2}}{a\left( 2\left\vert s\right\vert
+a^{-1}\left( \left\vert y\right\vert \right) \right) ^{2}}\int_{2\left\vert
s\right\vert +a^{-1}\left( \left\vert y\right\vert \right) }^{\infty }\left(
t-s\right) ^{-1}\gamma \left( \frac{t-s}{2\left\vert s\right\vert
+a^{-1}\left( \left\vert y\right\vert \right) }\right) ^{-2}dt \\
&\leq &C\int_{1/2}^{\infty }\gamma \left( r\right) ^{-2}\frac{dr}{r},
\end{eqnarray*}

because 
\begin{equation*}
\frac{\left\vert y\right\vert }{a\left( 2\left\vert s\right\vert
+a^{-1}\left( \left\vert y\right\vert \right) \right) }\leq \frac{a\left(
a^{-1}\left( \left\vert y\right\vert \right) +\right) }{a\left( 2\left\vert
s\right\vert +a^{-1}\left( \left\vert y\right\vert \right) \right) }\leq 1.
\end{equation*}

Hence, $\mathcal{I}_{2,1}\leq C$. Since

\begin{eqnarray*}
\mathcal{I}_{2,2} &\leq &\int_{2\left\vert s\right\vert }^{\infty }\left[
\int \chi _{Q_{C_{0}\delta }^{c}\left( 0\right) }\left\vert L^{\mu
;1/2}p^{\pi ^{\ast }}\left( t-s,x\right) -L^{\mu ;\frac{1}{2}}p^{\pi ^{\ast
}}\left( t,x\right) \right\vert dx\right] ^{2}dt \\
&+&\int_{\varepsilon \vee \left\vert s\right\vert }^{\varepsilon +\left\vert
s\right\vert }\left[ \int \left\vert L^{\mu ;\frac{1}{2}}p^{\pi ^{\ast
}}\left( t,x\right) \right\vert dx\right] ^{2}dt \\
&=&\mathcal{I}_{2,2,1}+\mathcal{I}_{2,2,2},
\end{eqnarray*}%
it follows by Lemma \ref{mvt}b) that%
\begin{eqnarray*}
\mathcal{I}_{2,2,1} &\leq &\int_{2\left\vert s\right\vert }^{\infty }\left[
\int \chi _{Q_{C_{0}\delta }^{c}\left( 0\right) }\left\vert L^{\mu
;1/2}p^{\pi ^{\ast }}\left( t-s,x\right) -L^{\mu ;\frac{1}{2}}p^{\pi ^{\ast
}}\left( t,x\right) \right\vert dx\right] ^{2}dt \\
&\leq &C.
\end{eqnarray*}

Clearly, $\mathcal{I}_{2,2,2}\leq C$ as well. Hence (\ref{hc}) and (\ref{ff2}%
) are proved. Combining the estimates of $I_{1}$ and $I_{2}$, we deduce that
the claim of Lemma \ref{lem-smoothEst2} holds.

\subsection{ Proof of Theorem \protect\ref{t1}}

We finish the proof of Theorem \ref{t1} in a standard way. Since by
Proposition 1 and 2 of \cite{Mph2}, $J_{\mu }^{t}:\mathbb{H}_{p}^{\mu
;s}\left( \mathbf{R}^{d},l_{2}\right) \rightarrow \mathbb{H}_{p}^{\mu
;s-t}\left( \mathbf{R}^{d},l_{2}\right) $ and $J_{\mu }^{t}:\mathbb{B}%
_{pp}^{\mu ,N;s}\left( \mathbf{R}^{d}\right) \rightarrow \mathbb{B}%
_{pp}^{\mu ,N;s-t}\left( \mathbf{R}^{d}\right) $ is an isomorphism for any $%
s,t\in \mathbf{R}$, it is enough to derive the statement for $s=0.$ Let $%
f\in \mathbb{L}_{p}\left( E\right) ,g\in \mathbb{B}_{pp}^{\mu
,N;1-1/p}\left( \mathbf{R}^{d}\right) ,$ and 
\begin{eqnarray*}
\Phi &\in &\mathbb{B}_{p,pp}^{\mu ,N;1-1/p}\left( E\right) \cap \mathbb{H}%
_{2,p}^{\mu ;\frac{1}{2}}\left( E\right) \text{ if }p\geq 2, \\
\Phi &\in &\mathbb{B}_{p,pp}^{\mu ,N;1-1/p}\left( E\right) \text{ if }p\in
\left( 1,2\right) .
\end{eqnarray*}

According to Lemma \ref{lem-denseSubspace}, there are sequences $f_{n}\in 
\tilde{\mathbb{C}}_{0,p}^{\infty }\left( E\right) ,g_{n}\in \tilde{\mathbb{C}%
}_{0,p}^{\infty }\left( \mathbf{R}^{d}\right) ,\Phi _{n}\in \tilde{\mathbb{C}%
}_{2,p}^{\infty }\left( E\right) \cap \tilde{\mathbb{C}}_{p,p}^{\infty
}\left( E\right) $ if $p\geq 2$, and $\Phi _{n}\in \tilde{\mathbb{C}}%
_{p,p}^{\infty }\left( E\right) $ if $p\in \left( 1,2\right) $, such that 
\begin{equation*}
f_{n}\rightarrow f\text{ in }\mathbb{L}_{p}\left( E\right) ,g_{n}\rightarrow
g\text{ in }\mathbb{B}_{pp}^{\mu ,N;1-1/p}\left( \mathbf{R}^{d}\right) ,
\end{equation*}%
and%
\begin{eqnarray*}
\Phi _{n} &\rightarrow &\Phi \text{ in }\mathbb{B}_{p,pp}^{\mu
,N;1-1/p}\left( E\right) \cap \mathbb{H}_{2,p}^{\mu ;\frac{1}{2}}\left(
E\right) \text{ if }p\geq 2, \\
\Phi _{n} &\rightarrow &\Phi \text{ in }\mathbb{B}_{p,pp}^{\mu
,N;1-1/p}\left( E\right) \text{ if }p\in \left( 1,2\right) .
\end{eqnarray*}%
For each $n,$ there is unique $u_{n}\in \tilde{\mathbb{C}}_{0,p}^{\infty
}\left( E\right) $ solving (\ref{mainEq}). Hence for $u_{n,m}=u_{n}-u_{m,}$
we have 
\begin{eqnarray*}
\partial _{t}u_{n,m} &=&\left( L^{\pi }-\lambda \right)
u_{n,m}+f_{n}-f_{m}+\int_{U}\left( \Phi _{n}-\Phi _{m}\right) q\left(
dt,d\nu \right) , \\
u_{n,m}\left( 0,x\right)  &=&g_{n}\left( x\right) -g_{m}\left( x\right)
,x\in \mathbf{R}^{d}.
\end{eqnarray*}%
By Lemma \ref{lem:basicEst2}, (4.15), (4.20) in \cite{Mph2}, and Lemma \ref%
{lem-smoothEst2}, for $p\geq 2,$ 
\begin{eqnarray*}
&&\left\vert L^{\mu }u_{n,m}\right\vert _{\mathbb{L}_{p}\left( E\right)
}\leq C[\left\vert f_{n}-f_{m}\right\vert _{\mathbb{L}_{p}\left( E\right)
}+\left\vert g_{n}-g_{m}\right\vert _{\mathbb{B}_{pp}^{\mu ,N;1-1/p}\left( 
\mathbf{R}^{d}\right) } \\
&&+\left\vert \Phi _{n}-\Phi _{m}\right\vert _{\mathbb{B}_{p,pp}^{\mu
,N;1-1/p}\left( E\right) }+\left\vert \Phi _{n}-\Phi _{m}\right\vert _{%
\mathbb{H}_{2,p}^{\mu ;1/2}\left( E\right) }],
\end{eqnarray*}%
and 
\begin{eqnarray*}
&&\left\vert L^{\mu }u_{n,m}\right\vert _{\mathbb{L}_{p}\left( E\right)
}\leq C[\left\vert f_{n}-f_{m}\right\vert _{\mathbb{L}_{p}\left( E\right)
}+\left\vert g_{n}-g_{m}\right\vert _{\mathbb{B}_{pp}^{\mu ,N;1-1/p}\left( 
\mathbf{R}^{d}\right) } \\
&&+\left\vert \Phi _{n}-\Phi _{m}\right\vert _{\mathbb{B}_{p,pp}^{\mu
,N;1-1/p}\left( E\right) }
\end{eqnarray*}%
if $p\in \left( 1,2\right) $. 

By Lemma \ref{lem-basicLpEst}, 
\begin{eqnarray*}
\left\vert u_{n,m}\right\vert _{\mathbb{L}_{p}\left( E\right) } &\leq
&C[\rho _{\lambda }\left\vert f_{n}-f_{m}\right\vert _{\mathbb{L}_{p}\left(
E\right) }+\rho _{\lambda }^{1/p}\left\vert g_{n}-g_{m}\right\vert _{\mathbb{%
L}_{p}\left( \mathbf{R}^{d}\right) } \\
&&+\rho _{\lambda }^{1/p}\left\vert \Phi _{n}-\Phi _{m}\right\vert _{\mathbb{%
L}_{p,p}\left( E\right) }+\rho _{\lambda }^{1/2}\left\vert \Phi _{n}-\Phi
_{m}\right\vert _{\mathbb{L}_{2,p}\left( E\right) }]
\end{eqnarray*}%
if $p\geq 2$, and 
\begin{eqnarray*}
\left\vert u_{n,m}\right\vert _{\mathbb{L}_{p}\left( E\right) } &\leq
&C[\rho _{\lambda }\left\vert f_{n}-f_{m}\right\vert _{\mathbb{L}_{p}\left(
E\right) }+\rho _{\lambda }^{1/p}\left\vert g_{n}-g_{m}\right\vert _{\mathbb{%
L}_{p}\left( \mathbf{R}^{d}\right) } \\
&&+\rho _{\lambda }^{1/p}\left\vert \Phi _{n}-\Phi _{m}\right\vert _{\mathbb{%
L}_{p,p}\left( E\right) }]
\end{eqnarray*}%
if $p\in \left( 1,2\right) $. Hence there is $u\in \mathbb{H}_{p}^{\mu
;1}\left( E\right) $ so that $u_{n}\rightarrow u$ in $\mathbb{H}_{p}^{\mu
;1}\left( E\right) $. Moreover, by Lemma \ref{lem-basicLpEst}, 
\begin{equation}
\sup_{t\leq T}\left\vert u_{n}\left( t\right) -u\left( t\right) \right\vert
_{\mathbb{L}_{p}\left( \mathbf{R}^{d}\right) }\rightarrow 0,  \label{fo7}
\end{equation}%
and $u$ is $L_{p}\left( \mathbf{R}^{d}\right) $-valued continuous. According
to Lemma 14 of \cite{Mph2}, 
\begin{equation}
\left\vert L^{\pi }f\right\vert _{\mathbb{L}_{p}\left( E\right) }\leq
C\left\vert L^{\mu }f\right\vert _{\mathbb{L}_{p}\left( E\right) },f\in 
\tilde{\mathbb{C}}_{0,p}^{\infty }\left( E\right) .  \label{fo8}
\end{equation}

By Lemma \ref{lem-stochInt} (see Appendix) and Remark \ref{re0}, 
\begin{equation}
\sup_{t\leq T}\left\vert \int_{0}^{t}\int_{U}\Phi _{n}q\left( ds,dz\right)
-\int_{0}^{t}\int_{U}\Phi q\left( ds,dz\right) \right\vert _{L_{p}\left( 
\mathbf{R}^{d}\right) }\rightarrow 0  \label{fo80}
\end{equation}%
as $n\rightarrow \infty $ in probability.

Hence (see (\ref{fo7})-(\ref{fo80})) we can pass to the limit in the
equation 
\begin{equation}
u_{n}\left( t\right) =g_{n}+\int_{0}^{t}[L^{\pi }u_{n}\left( s\right)
-\lambda u_{n}\left( s\right) +f_{n}\left( s\right)
]ds+\int_{0}^{t}\int_{U}\Phi _{n}q\left( ds,dz\right) ,0\leq t\leq T.
\label{fo9}
\end{equation}%
Obviously, (\ref{fo9}) holds for $u,g$ and $f,\Phi $. We proved the
existence part of Theorem \ref{t1}.

\emph{Uniqueness. }Assume $u_{1},u_{2}\in \mathbb{H}_{p}^{\mu ;1}\left(
E\right) $ solve (\ref{mainEq})$.$ Then $u=u_{1}-u_{2}\in \mathbb{H}%
_{p}^{\mu ;1}\left( E\right) $ solves (\ref{mainEq}) with $f=0,g=0,\Phi =0$.
Thus the uniqueness follows from the uniqueness of a deterministic equation
(see \cite{Mph2}).

Theorem \ref{t1} is proved.

\section{Appendix}

\subsection{Stochastic Integral}

We discuss here the definition of stochastic integrals with respect to a
martingale measure. Let $\left( \Omega ,\mathcal{F},\mathbf{P}\right) $ be a
complete probability space with a filtration of $\sigma -$algebras on $%
\mathbb{F}=\left( \mathcal{F}_{t},t\geq 0\right) $ satisfying the usual
conditions. Let $\left( U,\mathcal{U},\Pi \right) $ be a measurable space
with $\sigma -$finite measure $\Pi $, $\mathbf{R}_{0}^{d}=\mathbf{R}%
^{d}\backslash \left\{ 0\right\} $. Let $p\left( dt,d\nu \right) $ be $%
\mathbb{F}-$adapted point measures on $\left( \left[ 0,\infty \right) \times
U,\mathcal{B}\left( \left[ 0,\infty \right) \right) \otimes \mathcal{U}%
\right) $ with compensator $\Pi \left( d\nu \right) dt.$ We denote the
martingale measure $q\left( dt,d\nu \right) =p\left( dt,d\nu \right) -\Pi
\left( d\nu \right) dt$.

We prove the following based on Lemma 12 from \cite{MP1}.

\begin{lemma}
\label{lem-stochInt}Let $s\in \mathbf{R}$, $\Phi \in \mathbb{H}%
_{2,p}^{s}\left( E\right) \cap \mathbb{H}_{p,p}^{s}\left( E\right) $ for $%
p\in \left[ 2,\infty \right) $ and $\Phi \in \mathbb{H}_{p,p}^{s}\left(
E\right) $ for $p\in \left[ 1,2\right) $. There is a unique cadlag $%
H_{p}^{s}\left( \mathbf{R}^{d}\right) -$valued process 
\begin{equation*}
M\left( t\right) =\int_{0}^{t}\int \Phi \left( r,x,\nu \right) q\left(
dr,d\nu \right) ,0\leq t\leq T,x\in \mathbf{R}^{d},
\end{equation*}

such that for every $\varphi \in \mathcal{S}\left( \mathbf{R}^{d}\right) $, 
\begin{equation}
\left\langle M\left( t\right) ,\varphi \right\rangle =\int_{0}^{t}\int
\left( \int J^{s}\Phi \left( r,\cdot ,\nu \right) J^{-s}\varphi dx\right)
q\left( dr,d\nu \right) ,0\leq t\leq T.  \label{eq:unique}
\end{equation}

Moreover, there is a constant $C$ independent of $\Phi $ such that 
\begin{eqnarray}
&&\mathbf{E}\sup_{t\leq T}\left\vert \int_{0}^{t}\int \Phi \left( r,\cdot
,\nu \right) q\left( dr,d\nu \right) \right\vert _{H_{p}^{s}\left( \mathbf{R}%
^{d}\right) }  \label{eq:normEst} \\
&\leq &C\sum_{j=2,p}\left\vert \Phi \right\vert _{\mathbb{H}_{j,p}^{s}\left(
E\right) },p\geq 2,  \notag \\
&&\mathbf{E}\sup_{t\leq T}\left\vert \int_{0}^{t}\int \Phi \left( r,\cdot
,\nu \right) q\left( dr,d\nu \right) \right\vert _{H_{p}^{s}\left( \mathbf{R}%
^{d}\right) }  \notag \\
&\leq &C\left\vert \Phi \right\vert _{\mathbb{H}_{p,p}^{s}\left( E\right)
},p\in \left( 1,2\right) .  \notag
\end{eqnarray}
\end{lemma}

\begin{proof}
According to Proposition 2 in \cite{Mph2}, it is enough to consider the case 
$s=0$. Let $\Phi _{n}$ be a sequence defined in Lemma \ref%
{lem-approximation1} that approximates $\Phi $. Note first that by Lemma \ref%
{lem-approximation1}, for all $x$, 
\begin{equation*}
\mathbf{E}\int_{0}^{T}\sup_{x}\int \left\vert D^{\gamma }\Phi _{n}\left(
r,x,\nu \right) \right\vert ^{p}\Pi \left( d\nu \right) dr<\infty .
\end{equation*}%
Recall for each $n$, $\ \,$we have $\Phi _{n}=\Phi _{n}\chi _{U_{k}}$ for
some $U_{k}\in \mathcal{U}$ with $\Pi \left( U_{k}\right) <\infty $.
Consequently, we define for each $x\in \mathbf{R}^{d}$ and $\mathbf{P}$-a.s.
for all $\left( t,x\right) \in E,$ 
\begin{eqnarray*}
M_{n}\left( t,x\right) &=&\int_{0}^{t}\int \Phi _{n}\left( r,x,z\right)
q\left( dr,dz\right) \\
&=&\int_{0}^{t}\int \Phi _{n}\left( r,x,z\right) p\left( dr,dz\right)
-\int_{0}^{t}\int \Phi _{n}\left( r,x,z\right) \Pi \left( dz\right) dr.
\end{eqnarray*}%
Obviously, \ $M^{n}\left( t,x\right) $ is cadlag in $t$ and infinitely
differentiable in $x$. Obviously,$M_{n}\left( t\right) =M_{n}\left( t,\cdot
\right) $ is $\mathbb{L}_{p}\left( \mathbf{R}^{d}\right) $-valued cadlag
and, according to \cite{mp2}, there is a constant $C$ independent of $\Phi
_{n}$ such that 
\begin{equation}
\mathbf{E}\sup_{t\leq T}\left\vert M_{n}\left( t\right) \right\vert
_{L_{p}\left( \mathbf{R}^{d}\right) }^{p}\leq C\mathbf{E}\sum_{j=2,p}\left%
\vert \Phi _{n}\right\vert _{L_{j,p}\left( E\right) }^{p},2\leq p<\infty .
\label{f1}
\end{equation}

By Lemma 9 of \cite{MF},

\begin{equation}
\mathbf{E}\sup_{t\leq T}\left\vert M_{n}\left( t\right) \right\vert
_{L_{p}\left( \mathbf{R}^{d}\right) }^{p}\leq C\mathbf{E}\left\vert \Phi
_{n}\right\vert _{L_{p,p}\left( E\right) }^{p},1<p<2.  \label{f2}
\end{equation}

In addition, by Fubini theorem, $\mathbf{P}$-a.s. for $0\leq t\leq T,\varphi
\in \mathcal{S}\left( \mathbf{R}^{d}\right) ,$ 
\begin{eqnarray}
\left\langle M_{n}\left( t\right) ,\varphi \right\rangle &=&\int M_{n}\left(
t,x\right) \varphi \left( x\right) dx  \label{eq:stoch3} \\
&=&\int_{0}^{t}\int \left( \int \Phi _{n}\left( r,x,\nu \right) \varphi
\left( x\right) dx\right) q\left( dr,d\nu \right) .  \notag
\end{eqnarray}

By Lemma \ref{lem-approximation1}, 
\begin{eqnarray*}
\mathbf{E}\sum_{j=2,p}\left\vert \Phi _{n}-\Phi \right\vert _{L_{j,p}\left( 
\mathbf{R}^{d}\right) }^{p} &\rightarrow &0,2\leq p, \\
\mathbf{E}\left\vert \Phi _{n}-\Phi \right\vert _{L_{p,p}\left( E\right)
}^{p} &\rightarrow &0,p\in (1,2).
\end{eqnarray*}

Similarly, for each $p\in \lbrack 2,\infty )$, 
\begin{equation*}
\mathbf{E}\sup_{t\leq T}\left\vert M_{n}\left( t\right) -M_{m}\left(
t\right) \right\vert _{L_{p}\left( \mathbf{R}^{d}\right) }^{p}\leq C\mathbf{E%
}\sum_{j=2,p}\left\vert \Phi _{n}-\Phi _{m}\right\vert _{L_{j,p}\left(
E\right) }^{p}\rightarrow 0,
\end{equation*}%
and for each $p\in \left( 1,2\right) $%
\begin{equation*}
\mathbf{E}\sup_{t\leq T}\left\vert M_{n}\left( t\right) -M_{m}\left(
t\right) \right\vert _{L_{p}\left( \mathbf{R}^{d}\right) }^{p}\leq C\mathbf{E%
}\left\vert \Phi _{n}-\Phi _{m}\right\vert _{L_{p,p}\left( E\right)
}^{p}\rightarrow 0,
\end{equation*}

as $n,m\rightarrow \infty $. Therefore there is an adapted cadlag $%
L_{p}\left( \mathbf{R}^{d}\right) $ -valued process $M\left( t\right) $ so
that 
\begin{equation*}
\mathbf{E}\sup_{t\leq T}\left\vert M_{n}\left( t\right) -M\left( t\right)
\right\vert _{L_{p}\left( \mathbf{R}^{d}\right) }^{p}\rightarrow 0
\end{equation*}

as $n\rightarrow \infty .$ Passing to the limit as $n\rightarrow \infty $ in
(\ref{f1}), (\ref{f2}) and (\ref{eq:stoch3}) we derive (\ref{eq:normEst}), (%
\ref{eq:unique}). Henceforth, we define $\int_{0}^{t}\int \Phi \left(
r,x,\nu \right) q\left( dr,d\nu \right) $ to be $M\left( t\right) $ in this
lemma.
\end{proof}

\subsection{Maximal and sharp functions, H\"{o}rmander condition}

Given a function $\kappa :\left( 0,\infty \right) \rightarrow \left(
0,\infty \right) $, consider the collection $\mathbb{Q}$ of sets $Q_{\delta
}=Q_{\delta }\left( t,x\right) =\left( t-\kappa \left( \delta \right)
,t+\kappa \left( \delta \right) \right) \times B_{\delta }\left( x\right)
,\left( t,x\right) \in \mathbf{R}\times \mathbf{R}^{d}=\mathbf{R}%
^{d+1},\delta >0$, where $B_{\delta }\left( x\right) $ is the standard open
ball of radius $\delta $ centered at $x$. The volume $\left\vert Q_{\delta
}\left( t,x\right) \right\vert =c_{0}\kappa \left( \delta \right) \delta
^{d}.$ We will need the following assumptions.

\textbf{A1. }$\kappa$ is continuous, $\lim_{\delta\rightarrow0}\kappa\left(%
\delta\right)=0$ and $\lim_{\delta\rightarrow\infty}\kappa\left(\delta%
\right)=\infty.$

\textbf{A2.} There is a nondecreasing continuous function $%
l\left(\varepsilon\right),\varepsilon>0,$ such that $\lim_{\varepsilon%
\rightarrow0}l\left(\varepsilon\right)=0$ and 
\begin{equation*}
\kappa\left(\varepsilon r\right)\leq
l\left(\varepsilon\right)\kappa(r),r>0,\varepsilon>0.
\end{equation*}

Since $Q_{\delta}\left(t,x\right)$ not exactly increases in $\delta$, we
present the basic estimates involving maximal functions based on the system $%
\mathbb{Q=}\left\{ Q_{\delta}\right\} $.

We state the following engulfing property from \cite{MPh}.

\begin{lemma}
\label{le1}Let \textbf{A}2\textbf{\ }hold. If $Q_{\delta}\left(t,x\right)%
\cap Q_{\delta^{\prime}}\left(r,z\right)\neq\emptyset$ with $%
\delta^{\prime}\leq\delta$, then there is $K_{0}\geq3$ such that $%
Q_{K_{0}\delta}\left(t,x\right)$ contains both, $Q_{\delta}\left(t,x\right)$
and $Q_{\delta^{\prime}}\left(r,z\right),$ and 
\begin{equation*}
\left\vert Q_{\delta}\left(t,x\right)\right\vert \leq\left\vert
Q_{K_{0}\delta}\left(t,x\right)\right\vert \leq
K_{0}^{d}l\left(K_{0}\right)\left\vert Q_{\delta}\left(t,x\right)\right\vert
.
\end{equation*}
\end{lemma}

\subsubsection{Maximal and sharp functions}

Following \cite{stein1}, for a locally integrable function $%
f\left(t,x\right) $ on $\mathbf{R}^{d+1}$ we define 
\begin{equation*}
\left(A_{\delta}f\right)(t,x)=\frac{1}{\left\vert
Q_{\delta}\left(t,x\right)\right\vert }\int_{Q_{\delta}\left(t,x\right)}f%
\left(s,y\right)dsdy,\left(t,x\right)\in\mathbf{R\times R}^{d},\delta>0
\end{equation*}
and the maximal function of $f$ by 
\begin{equation*}
\mathcal{M}f\left(t,x\right)=\sup_{\delta>0}\left(A_{\delta}\left\vert
f\right\vert \right)(t,x),\left(t,x\right)\in\mathbf{R}^{d+1}.
\end{equation*}
We use collection $\mathbb{Q}$ to define a larger, noncentered maximal
function of $f$, as 
\begin{equation*}
\widetilde{\mathcal{M}}f\left(t,x\right)=\sup_{\left(t,x\right)\in Q}\frac{1%
}{\left\vert Q\right\vert }\int_{Q}|f\left(s,y\right)|dsdy,\left(t,x\right)%
\in\mathbf{R}^{d+1},
\end{equation*}
where $\sup$ is taken over all $Q\in\mathbb{Q}$ that contain $(t,x).$

\begin{remark}
\label{re2}It is shown in \cite{MPh} that if \textbf{A2} hold then there
exists $K_{0}>0$ such that for a locally integrable $f$ on $\mathbf{R}%
^{d+1}, $ 
\begin{equation*}
\mathcal{M}f\leq\widetilde{\mathcal{M}}f\leq\frac{1}{K_{0}^{d}l\left(K_{0}%
\right)}\mathcal{M}f.
\end{equation*}
\end{remark}

The following result is proved in \cite{MPh}.

\begin{theorem}
\label{ft}Let \textbf{A2} hold and $f$ be measurable function on $\mathbf{R}%
^{d+1}=\mathbf{R\times R}^{d}.$

(a) If $f\in L_{p},1\leq p\leq\infty$, then $\mathcal{M}f$ is finite a.e.

(b) If $f\in L_{1}$, then for every $\alpha>0$, 
\begin{equation*}
\left\vert \left\{ \mathcal{M}f\left(t,x\right)>\alpha\right\} \right\vert
\leq\frac{c}{\alpha}\int|f|dtdx.
\end{equation*}

(c) If $f\in L_{p},1<p\leq \infty $, then $\mathcal{M}f\in L_{p}$ and 
\begin{equation*}
\left\vert \mathcal{M}f\right\vert _{L_{p}}\leq N_{p}\left\vert f\right\vert
_{L_{p}},
\end{equation*}%
where $N_{p}$ depends only on $p,l$ and $K_{0}.$
\end{theorem}

\paragraph{Calderon-Zygmund decomposition, sharp functions}

Assume \textbf{A1, A2} hold. Let $F\subseteq \mathbf{R\times R}^{d}$ be
closed and $O=F^{c}=\mathbf{R}^{d+1}\backslash F.$ For $\left( t,x\right)
\in O$, let 
\begin{equation*}
D\left( t,x\right) =\inf \left\{ \delta >0:Q_{\delta }\left( t,x\right) \cap
F\neq \emptyset \right\} .
\end{equation*}
In \cite{MPh}, the following statement is proved.

\begin{lemma}
\label{wh}(Lemma 15 in \cite{MPh}) Assume \textbf{A1, A2} hold. Given a
closed nonempty $F$, there are sequences $Q^{k}$, $Q^{\ast k}$ and $Q^{\ast
\ast k}$ in $\mathbb{Q\,\ }$having the same center but with radius expanded
by the same factor $c_{1}^{\ast \ast }>c_{1}^{\ast }>c_{1}\,\ $so that $%
Q^{k}\subseteq Q^{\ast k}\subseteq Q^{\ast \ast k}$ (all of them are of the
form $Q_{bD\left( t_{k},x_{k}\right) }\left( t_{k},x_{k}\right) $ with $%
b=c_{1},c_{1}^{\ast },c_{1}^{\ast \ast }$ correspondingly) and

(a) the sets $Q^{k}$ are disjoint.

(b) $\cup_{k}Q^{\ast k}=O=F^{c}.$

(c) $Q^{\ast\ast k}\cap F\neq\emptyset$ for each $k$.
\end{lemma}

\begin{remark}
\label{re3}Assume \textbf{A1, A2} hold and $Q^{k}\subseteq Q^{\ast
k}\subseteq Q^{\ast\ast k}$ be the sequences in $\mathbb{Q}$ from Lemma \ref%
{wh}. It is easy to find a sequence of disjoint measurable sets $C^{k}$ so
that $Q^{k}\subseteq C^{k}\subseteq Q^{\ast k}$ and $\cup_{k}C^{k}=O$. For
example (see Remark, p. 15, in \cite{stein1}), 
\begin{equation*}
C^{k}=Q^{\ast
k}\cap\left(\cup_{j<k}C^{j}\right)^{c}\cap\left(\cup_{j>k}Q^{j}\right)^{c}.
\end{equation*}
\end{remark}

We have the following Calderon-Zygmund decomposition for $\mathbb{Q}.$

\begin{theorem}
\label{whit}(Theorem 4 in \cite{MPh}) Assume \textbf{A1, A2} hold. Let $f\in
L^{1}\left( \mathbf{R\times R}^{d}\right) $, $\alpha >0$ and $O_{\alpha
}=\left\{ \widetilde{\mathcal{M}}f>\alpha \right\} .$ Consider the sets $%
Q^{k}\subseteq C^{k}\subseteq Q^{\ast k}\subseteq O$ of Lemma \ref{wh} and
Remark \ref{re3} associated to $O_{\alpha }.$

There is a decomposition $f=g+b$ with 
\begin{equation}
g\left(x\right)=\left\{ 
\begin{array}{cc}
f(x) & \text{if }x\notin O_{\alpha}, \\ 
\frac{1}{\left\vert C^{k}\right\vert }\int_{C^{k}}f & \text{ if }x\in
C^{k},k\geq1,%
\end{array}%
\right.  \label{2}
\end{equation}
and with $b=\sum_{k}b_{k}$, where 
\begin{equation}
b_{k}=\chi_{C^{k}}\left[f\left(x\right)-\frac{1}{\left\vert C^{k}\right\vert 
}\int_{C^{k}}f\text{ }\right],k\geq1,  \label{3}
\end{equation}
(note $C^{k}$ are disjoint, $\cup_{k}C^{k}=O_{\alpha}$). Also,

(i) $\left\vert g\left(x\right)\right\vert \leq c\alpha$ for a.e. $x.$

(ii) support($b_{k})\subseteq Q^{\ast k},$ 
\begin{equation*}
\int b_{k}=0\text{ and }\int\left\vert b_{k}\right\vert \leq
c\alpha\left\vert Q^{\ast k}\right\vert .
\end{equation*}

(iii) $\sum_{k}\left\vert Q^{\ast k}\right\vert \leq\frac{c}{\alpha}%
\int\left\vert f\right\vert .$
\end{theorem}

The $L^{p}$ norm of $f$ can be controlled by its sharp function as well.

\begin{definition}
Given a locally integrable $f$ on $\mathbf{R\times R}^{d}$, we define its
sharp function as 
\begin{equation*}
f^{\natural}\left(t,x\right)=\sup_{\delta}\frac{1}{\left\vert
Q_{\delta}\left(t,x\right)\right\vert }\int_{Q_{\delta}\left(t,x\right)}%
\left\vert f\left(s,y\right)-f_{Q_{\delta}\left(t,x\right)}\right\vert dsdy.
\end{equation*}
or

\begin{equation*}
f^{\sharp }\left( t,x\right) =\sup_{\left( t,x\right) \in Q}\frac{1}{%
\left\vert Q\right\vert }\int_{Q}\left\vert f\left( s,y\right)
-f_{Q}\right\vert dsdy,
\end{equation*}%
where 
\begin{equation*}
f_{Q}=\frac{1}{Q}\int_{Q}fdm,
\end{equation*}%
and $\sup $ is taken over all $Q\in \mathbb{Q}$ containing $(t,x)$.
\end{definition}

Obviously, 
\begin{equation*}
f^{\sharp }\left( t,x\right) \leq 2\widetilde{\mathcal{M}}f\left( t,x\right)
,f^{\natural }\leq 2\mathcal{M}f.
\end{equation*}

\begin{remark}
\label{re4}Let $f\in L_{loc}^{1}$, $B\subseteq\mathbf{R}^{d}$ be a bounded
measurable subset. Then for any constant $C,$ 
\begin{eqnarray*}
& & \frac{1}{\left\vert B\right\vert ^{2}}\int_{B}\int_{B}\left\vert
f(s,y)-f\left(t,z\right)\right\vert dtdzdsdy \\
& \leq & \frac{1}{\left\vert B\right\vert ^{2}}\int_{B}%
\int_{B}|f(s,y)-C|dtdzdsdy+\frac{1}{\left\vert B\right\vert ^{2}}%
\int_{B}\int_{B}\left\vert C-f\left(t,z\right)\right\vert dtdzdsdy \\
& \leq & \frac{2}{\left\vert B\right\vert }\int_{B}\left\vert
f\left(t,x\right)-C\right\vert dtdx.
\end{eqnarray*}
Hence 
\begin{eqnarray}
\frac{1}{\left\vert B\right\vert }\int_{B}\left\vert f-f_{B}\right\vert &
\leq & \frac{1}{\left\vert B\right\vert ^{2}}\int_{B}\int_{B}\left\vert
f(s,y)-f\left(t,z\right)\right\vert dtdzdsdy  \label{4-1} \\
& \leq & 2\frac{1}{\left\vert B\right\vert }\int_{B}\left\vert
f-f_{B}\right\vert  \notag
\end{eqnarray}
\end{remark}

As a consequence of (\ref{4-1}), the following holds.

\begin{remark}
\label{re5}1. As in the case of maximal functions, $f^{\natural }\left(
t,x\right) \leq 2f^{\sharp }\left( t,x\right) \leq \frac{4}{\left(
K_{0}^{d}l\left( K_{0}\right) \right) ^{2}}f^{\natural }\left( t,x\right) $
or 
\begin{equation*}
\frac{1}{2}f^{\natural }\leq f^{\sharp }\leq \frac{2}{\left(
K_{0}^{d}l\left( K_{0}\right) \right) ^{2}}f^{\natural }.
\end{equation*}

2. Since $\left\vert \left\vert a\right\vert -\left\vert b\right\vert
\right\vert \leq\left\vert a-b\right\vert $, it follows by (\ref{4-1}) that $%
\left(\left\vert f\right\vert
\right)^{\natural}\leq2f^{\natural},\left(\left\vert f\right\vert
\right)^{\#}\leq2f^{\#}.$
\end{remark}

\begin{lemma}
\label{le2}(cf. Lemma 9, p.101, in \cite{kry1}) Let $\lambda=\frac{1}{2c}$ $%
(c$ is from Theorem \ref{whit})$,f\in L^{1}$ and $\alpha>0$. Then 
\begin{equation}
\left\vert \left\{ \left\vert f\right\vert >\alpha\right\} \right\vert \leq%
\frac{4}{\alpha}\int\chi_{\left\{ \widetilde{\mathcal{M}}f>\lambda\alpha%
\right\} }f^{\sharp}.  \label{1}
\end{equation}
If $f\geq0$, we $4/\alpha$ can be replaced by $2/\alpha$.
\end{lemma}

\begin{proof}
Assume $f\geq0$. Apply Theorem \ref{whit} with $\alpha$ replaced by $%
\lambda\alpha=\alpha/2c$, i.e. $\lambda=1/(2c)$. We have $f=g+b$ with $%
\left\vert g\right\vert \leq c\lambda\alpha=\alpha/2$ a.e. and 
\begin{equation*}
b=\sum_{k}b_{k}=\sum_{k}\chi_{C^{k}}\left[f\left(x\right)-\frac{1}{%
\left\vert C^{k}\right\vert }\int_{C^{k}}f\text{ }\right].
\end{equation*}
Recall $O=\left\{ \widetilde{\mathcal{M}}f>\lambda\alpha\right\}
=\cup_{k}C^{k}$ and $C^{k}$ are disjoint. Since $\left\{ \left\vert
f\right\vert >\alpha\right\} \subseteq\left\{ \left\vert b\right\vert
>\alpha/2\right\} ,$(\ref{1}) follows.
\end{proof}

\begin{theorem}
\label{fs}(Fefferman/Stein) Let $f\in L^{p}\left(\mathbf{R\times R}%
^{d}\right),p\in\left(1,\infty\right)$. Then 
\begin{equation*}
\left\vert f\right\vert _{L^{p}}\leq N\left\vert f^{\sharp}\right\vert
_{L^{p}}.
\end{equation*}
\end{theorem}

\begin{proof}
Indeed, using Lemma \ref{le2}, (\ref{1}), Theorem \ref{ft} and H\"{o}lder
inequality, for $f\in L^{1}$ we have 
\begin{eqnarray*}
\left\vert f\right\vert _{L^{p}}^{p} &=&\int_{0}^{\infty }\left\vert \left\{
|f|>\alpha ^{1/p}\right\} \right\vert d\alpha \leq 4\int_{0}^{\infty }\alpha
^{-1/p}\int \chi _{\left\{ \widetilde{\mathcal{M}}f>\lambda \alpha
^{1/p}\right\} }f^{\sharp } \\
&=&c\int \int_{0}^{(\widetilde{\mathcal{M}}f)^{p}\lambda ^{-p}}\alpha
^{-1/p}d\alpha f^{\sharp }=c\int (\widetilde{\mathcal{M}}f)^{p-1}f^{\sharp
}\leq N\left\vert \widetilde{\mathcal{M}}f\right\vert
_{L^{p}}^{p-1}\left\vert f^{\sharp }\right\vert _{L^{p}}.
\end{eqnarray*}%
If in addition $f\in L^{p}$, then we are done. If only $f\in L^{p}$, then we
take a sequence $f_{n}\in L^{1}\cap L^{p}$ converging to $f$ in $L^{p}$ and
notice that 
\begin{equation*}
f_{n}^{\sharp }\leq (f-f_{n})^{\sharp }+f^{\sharp }
\end{equation*}
and 
\begin{equation*}
\left\vert (f-f_{n})^{\sharp }\right\vert _{L^{p}}\leq 2\left\vert \mathcal{M%
}(f-f_{n})\right\vert _{L^{p}}\leq C\left\vert f-f_{n}\right\vert _{L^{p}}.
\end{equation*}
\end{proof}

\subsubsection{H\"{o}rmander Condition and $L_{p}$-estimate}

Let $V$ be a separable Hilbert space, let $f$ be a measurable $V-$valued
function on $\mathbf{R}^{d+1}$, define an operator $\mathcal{G}$ by 
\begin{equation}
\left( \mathcal{G}f\right) \left( t,x\right) =\left[ \int \left\vert \int_{%
\mathbf{R}^{d}}K\left( t,x,s,y\right) f\left( s,y\right) dy\right\vert
_{V}^{2}ds\right] ^{1/2},\left( t,x\right) \in \mathbf{R}^{d+1}
\label{eq:operatorG}
\end{equation}%
Let $K$ be a measurable function and for almost all $\left( t,x\right) \in 
\mathbf{R}^{d+1}$ the function $K\left( t,x,\cdot \right) f$ is integrable
for all $f\in C_{0}^{\infty }\left( \mathbf{R}^{d+1},V\right) $.

We assume that $\mathcal{G}$ is bounded on $L_{2}$, i.e., 
\begin{equation}
\left\vert \mathcal{G}f\right\vert _{L_{2}}\leq M_{0}\left\vert f\right\vert
_{L_{2}\left( \mathbf{R}^{d+1};V\right) },f\in L_{2}\left( \mathbf{R}%
^{d+1};V\right)  \label{eq:L2cond}
\end{equation}%
In \cite{kry} and \cite{mp2}, an $L_{p}$-estimate for $\mathcal{G}f$ was
derived by estimating directly its sharp function $\left( \mathcal{G}%
f\right) ^{\#}$. It was shown in \cite{KK2}, that $\left( \mathcal{G}%
f\right) ^{\#}$-estimate follows by verifying a H\"{o}rmander condition. We
adjust \cite{KK2} to our setting, and show that the sharp function estimate
follows form the following H\"{o}rmander condition: there are constants $%
C_{0}>1,M_{1}>0$ so that for any $Q_{\delta }\left( t,x\right) \in \mathbb{Q}
$,

\begin{equation}
\int \left[ \int \chi _{Q_{C_{0}\delta }^{c}\left( t,x\right) }\left\vert
K\left( t,x,r,y\right) -K\left( \bar{t},\bar{x},r,y\right) \right\vert dy%
\right] ^{2}dr\leq M_{1},\hspace{1em}\forall \left( \bar{t},\bar{x}\right)
\in Q_{\delta }\left( t,x\right) .  \label{eq:stochHormander}
\end{equation}

\begin{lemma}
\label{stochHormander} Let $\mathbf{A}_{1}$, $\mathbf{A}_{2}$, (\ref%
{eq:L2cond}), (\ref{eq:stochHormander}) hold. Then $\mathcal{G}$ is bounded
in $L_{p}-$norm on $L_{2}\cap L_{p}$ if $p>2$. More precisely 
\begin{equation*}
\left\vert \mathcal{G}f\right\vert _{L_{p}}\leq A_{p}\left\vert f\right\vert
_{L_{p}\left( \mathbf{R}^{d+1};V\right) },\quad \forall f\in L_{2}\left( 
\mathbf{R}^{d+1};V\right) \cap L_{p}\left( \mathbf{R}^{d+1};V\right) ,\quad
p>2,
\end{equation*}

where $A_{p}$ depends only on the constants $M_{1},M_{0}$ and $p$.
\end{lemma}

For convenience we denote,

\begin{equation*}
\mathcal{K}f\left( t,x,s\right) =\int_{\mathbf{R}^{d}}K\left( t,x,s,y\right)
f\left( s,y\right) dy,
\end{equation*}

and 
\begin{equation*}
Gf\left(t,x,s,y\right)=\left[\int\left|\mathcal{K}f\left(t,x,r\right)-%
\mathcal{K}\left(s,y,r\right)\right|_{V}^{2}dr\right]^{1/2}.
\end{equation*}

For the proof of Lemma \ref{stochHormander}, we will need some auxiliary
results.

\begin{lemma}
\label{lem3} Let $\left( t_{1},x_{1}\right) \in Q_{\delta }\left(
t_{0},x_{0}\right) $. Then for any $f_{1},f_{2}\in L_{2}\left( \mathbf{R}%
^{d+1};V\right) $, 
\begin{eqnarray*}
&&\fint_{Q_{\delta }\left( t_{0},x_{0}\right) }\fint_{Q_{\delta }\left(
t_{0},x_{0}\right) }\left\vert \mathcal{G}\left( f_{1}+f_{2}\right) \left(
t,x\right) -\mathcal{G}\left( f_{1}+f_{2}\right) \left( s,y\right)
\right\vert dtdxdsdy \\
&\leq &2\widetilde{\mathcal{M}}\left( \mathcal{G}f_{1}\right) \left(
t_{1},x_{1}\right) +\fint_{Q_{\delta }\left( t_{0},x_{0}\right) }\fint%
_{Q_{\delta }\left( t_{0},x_{0}\right) }Gf_{2}\left( t,s,x,y\right) dtdxdsdy
\end{eqnarray*}
\end{lemma}

\begin{proof}
Set $f=f_{1}+f_{2}$, and let $\left( t,x\right) ,\left( s,y\right) \in
Q_{\delta }\left( t_{0},x_{0}\right) $. Then 
\begin{align*}
& \left\vert \mathcal{G}f\left( t,x\right) -\mathcal{G}f\left( s,y\right)
\right\vert \\
& =\left\vert \left[ \int \left\vert \mathcal{K}f\left( t,x,r\right)
\right\vert _{V}^{2}dr\right] ^{1/2}-\left[ \int \left\vert \mathcal{K}%
f\left( s,y,r\right) \right\vert _{V}^{2}dr\right] ^{1/2}\right\vert \\
& \leq \left[ \int \left\vert \mathcal{K}f\left( t,x,r\right) -\mathcal{K}%
f\left( s,y,r\right) \right\vert _{V}^{2}dr\right] ^{1/2} \\
& \leq \left[ \int \left\vert \mathcal{K}f_{1}\left( t,x,r\right)
\right\vert _{V}^{2}dr\right] ^{1/2}+\left[ \int \left\vert \mathcal{K}%
f_{1}\left( s,y,r\right) \right\vert _{V}^{2}dr\right] ^{1/2} \\
& +\left[ \int \left\vert \mathcal{K}f_{2}\left( t,x,r\right) -\mathcal{K}%
f_{2}\left( s,y,r\right) \right\vert _{V}^{2}dr\right] ^{1/2}
\end{align*}

Taking average on $Q_{\delta }\left( t_{0},x_{0}\right) ,$ the result
follows.
\end{proof}

\begin{lemma}
\label{lem4} Suppose (\ref{eq:L2cond}) holds, $f$ belongs to $L_{2}\left( 
\mathbf{R}^{d+1};V\right) \cap L_{\infty }\left( \mathbf{R}^{d+1};V\right) $
and vanishes outside of $Q_{\gamma \delta }\left( t_{0},x_{0}\right) ,\gamma
>0$. Then 
\begin{equation*}
A:=\fint_{Q_{\delta }\left( t_{0},x_{0}\right) }\fint_{Q_{\delta }\left(
t_{0},x_{0}\right) }Gf\left( t,x,s,y\right) dtdxdsdy\leq C\left\vert
f\right\vert _{L_{\infty }\left( \mathbf{R}^{d+1},V\right) }
\end{equation*}

where $C=C\left( d,\gamma ,M_{0}\right) $.
\end{lemma}

\begin{proof}
Obviously, $Gf\left( t,s,x,y\right) \leq \mathcal{G}f\left( t,x\right) +%
\mathcal{G}f\left( s,y\right) $ for any $\left( t,x\right) ,\left(
s,y\right) \in Q_{\delta }\left( t_{0},x_{0}\right) .$ Hence, $A\leq 2\fint%
_{Q_{\delta }\left( t_{0},x_{0}\right) }\mathcal{G}f\left( t,x\right) dtdx$.

Applying H\"{o}lder's inequality and using (\ref{eq:L2cond}), 
\begin{eqnarray*}
\fint_{Q_{\delta }\left( t_{0},x_{0}\right) }\mathcal{G}f\left( t,x\right)
dtdx &\leq &\frac{1}{\left\vert Q_{\delta }\left( t_{0},x_{0}\right)
\right\vert ^{1/2}}\left[ \int \left\vert \mathcal{G}f\left( t,x\right)
\right\vert ^{2}dtdx\right] ^{1/2} \\
&\leq &\frac{M_{0}}{\left\vert Q_{\delta }\left( t_{0},x_{0}\right)
\right\vert ^{1/2}}\left[ \int \left\vert f\left( t,x\right) \right\vert
_{V}^{2}dtdx\right] ^{1/2} \\
&\leq &M_{0}\frac{\left\vert Q_{\gamma \delta }\left( t_{0},x_{0}\right)
\right\vert ^{1/2}}{\left\vert Q_{\delta }\left( t_{0},x_{0}\right)
\right\vert ^{1/2}}\left\vert f\right\vert _{L_{\infty }\left( \mathbf{R}%
^{d+1},V\right) }\leq M_{0}\left[ \gamma ^{d}l\left( \gamma \right) \right]
^{1/2}\left\vert f\right\vert _{L_{\infty }\left( \mathbf{R}^{d+1};V\right)
}.
\end{eqnarray*}
\end{proof}

\begin{lemma}
\label{lem5} Let $b\geq 2,l\left( b^{-1}\right) \leq 1/2$, $\gamma \geq
bC_{0}+1$ and $l\left( \gamma ^{-1}\right) ^{-1}\geq l\left( bC_{0}\right)
+1 $. Suppose that $f\in L_{\infty }\left( \mathbf{R}^{d+1};V\right) $
vanishes on $Q_{\gamma \delta }\left( t_{0},x_{0}\right) $ and (\ref%
{eq:stochHormander}) hold. Then, 
\begin{equation*}
\fint_{Q_{\delta }\left( t_{0},x_{0}\right) }\fint_{Q_{\delta }\left(
t_{0},x_{0}\right) }Gf\left( t,x,s,y\right) dtdxdsdy\leq C\left\vert
f\right\vert _{L_{\infty }\left( \mathbf{R}^{d+1};V\right) }
\end{equation*}

where $C=C\left(M_{1}\right)$.
\end{lemma}

\begin{proof}
If $\left( t,x\right) ,\left( s,y\right) \in Q_{\delta }\left(
t_{0},x_{0}\right) $, then $\left\vert x-y\right\vert <2\delta \leq b\delta $%
, and%
\begin{equation*}
\left\vert t-s\right\vert <2\kappa \left( \delta \right) \leq 2l\left( \frac{%
1}{b}\right) \kappa \left( b\delta \right) \leq \kappa \left( b\delta
\right) ,
\end{equation*}%
i.e. $\left( t,x\right) \in Q_{b\delta }\left( s,y\right) $. If $\left(
r,z\right) \in Q_{\gamma \delta }\left( t_{0},x_{0}\right) ^{c}$ then either 
$\left\vert z-x_{0}\right\vert \geq \gamma \delta $ or $\left\vert
r-t_{0}\right\vert \geq \kappa \left( \gamma \delta \right) $ .

If $\left\vert z-x_{0}\right\vert \geq \gamma \delta $ then $\left\vert
z-y\right\vert \geq \left\vert z-x_{0}\right\vert -\left\vert
y-x_{0}\right\vert \geq \gamma \delta -\delta \geq bC_{0}\delta $. On the
other hand, if $\left\vert r-t_{0}\right\vert \geq \kappa \left( \gamma
\delta \right) $ then%
\begin{equation*}
\left\vert r-s\right\vert \geq \left\vert r-t_{0}\right\vert -\left\vert
t_{0}-s\right\vert \geq \kappa \left( \gamma \delta \right) -\kappa \left(
\delta \right) \geq \kappa \left( \delta \right) \left[ l\left( \gamma
^{-1}\right) ^{-1}-1\right] \geq \kappa \left( \delta \right) l\left(
bC_{0}\right) \geq \kappa \left( bC_{0}\delta \right) .
\end{equation*}

Both cases imply that $\left( r,z\right) \in Q_{C_{0}b\delta }\left(
s,y\right) ^{c}$, and by the stochastic H\"{o}rmander condition (\ref%
{eq:stochHormander}) 
\begin{align*}
\left\vert Gf\left( t,x,s,y\right) \right\vert ^{2}& \leq \int \left[ \int_{%
\mathcal{\mathbf{R}}^{d}}\left\vert K\left( t,x,r,z\right) -K\left(
s,y,r,z\right) \right\vert \left\vert f\left( r,z\right) \right\vert _{V}dz%
\right] ^{2}dr \\
& \leq \left\vert f\right\vert _{L_{\infty }\left( \mathbf{R}^{d+1};V\right)
}^{2}\int \left[ \int \chi _{Q_{C_{0}b\delta }\left( s,y\right)
^{c}}\left\vert K\left( t,x,r,z\right) -K\left( s,y,r,z\right) \right\vert dz%
\right] ^{2}dr \\
& \leq M_{1}\left\vert f\right\vert _{L_{\infty }\left( \mathbf{R}%
^{d+1};V\right) }^{2}.
\end{align*}

Therefore, 
\begin{equation*}
\fint_{Q_{\delta}\left(t_{0},x_{0}\right)}\fint_{Q_{\delta}\left(t_{0},x_{0}%
\right)}Gf\left(t,s,x,y\right)dtdxdsdy\leq
M_{1}^{1/2}\left|f\right|_{L_{\infty}\left(E;V\right)}
\end{equation*}
\end{proof}

\begin{lemma}
\label{lem6}Let $f_{1}\in L_{2}\left( \mathbf{R}^{d+1};V\right) $, $f_{2}\in
L_{2}\left( \mathbf{R}^{d+1};V\right) \cap L_{\infty }\left( \mathbf{R}%
^{d+1};V\right) $ and suppose that (\ref{eq:L2cond}) and (\ref%
{eq:stochHormander}) hold. Then for any $X_{0}=\left( t_{0},x_{0}\right) \in 
\mathbf{R}^{d+1}$, 
\begin{equation*}
\left[ \mathcal{G}\left( f_{1}+f_{2}\right) \right] ^{\natural }\left(
t_{0},x_{0}\right) \leq 2\widetilde{\mathcal{M}}\left( \mathcal{G}%
f_{1}\right) \left( t_{0},x_{0}\right) +C\left\vert f_{2}\right\vert
_{L_{\infty }\left( \mathbf{R}^{d+1};V\right) }
\end{equation*}

where $C=C\left( d,M_{0},M_{1},C_{0}\right) $.
\end{lemma}

\begin{proof}
By Lemma \ref{lem3}, for any $\left( t_{0},x\,_{0}\right) \in \mathbf{R}%
^{d+1}$,%
\begin{align*}
& \fint_{Q_{\delta }\left( t_{0},x_{0}\right) }\fint_{Q_{\delta }\left(
t_{0},x_{0}\right) }\left\vert \mathcal{G}\left( f_{1}+f_{2}\right) \left(
t,x\right) -\mathcal{G}\left( f_{1}+f_{2}\right) \left( s,y\right)
\right\vert dtdxdsdy \\
\leq & 2\widetilde{\mathcal{M}}\left( \mathcal{G}f_{1}\right) \left(
t_{0},x_{0}\right) +\fint_{Q_{\delta }\left( t_{0},x_{0}\right) }\fint%
_{Q_{\delta }\left( t_{0},x_{0}\right) }Gf_{2}\left( t,s,x,y\right) dtdxdsdy.
\end{align*}

Moreover, defining $f_{2,1}\left( t,x\right) :=f_{2}\left( t,x\right) \chi
_{Q_{\gamma \delta }\left( t_{0},x_{0}\right) }\left( t,x\right) $ and $%
f_{2,2}\left( t,x\right) :=f_{2}\left( t,x\right) -f_{2,1}\left( t,x\right) $%
, we have 
\begin{align*}
& \fint_{Q_{\delta }\left( t_{0},x_{0}\right) }\fint_{Q_{\delta }\left(
t_{0},x_{0}\right) }Gf_{2}dtdxdsdy \\
\leq & \fint_{Q_{\delta }\left( t_{0},x_{0}\right) }\fint_{Q_{\delta }\left(
t_{0},x_{0}\right) }Gf_{2,1}dtdxdsdy+\fint_{Q_{\delta }\left(
t_{0},x_{0}\right) }\fint_{Q_{\delta }\left( t_{0},x_{0}\right)
}Gf_{2,2}dtdxdsdy.
\end{align*}

We obtain the results by (\ref{4-1}), and Lemmas \ref{lem4} and \ref{lem5}
and taking $\gamma $ satisfying the assumptions of Lemma \ref{lem5}.
\end{proof}

\paragraph{Proof of Lemma \protect\ref{stochHormander}}

Let $p>2$, $f\in L_{2}\left( \mathbf{R}^{d+1};V\right) \cap L_{p}\left( 
\mathbf{R}^{d+1};V\right) $. For $\lambda >0,\delta >0,$ we define $%
f=f_{1,\lambda }+f_{2,\lambda }$ with%
\begin{equation*}
f_{1,\lambda }=f1_{\left\vert f\right\vert _{V}>\delta \lambda
},f_{2,\lambda }=f1_{\left\vert f\right\vert _{V}\leq \delta \lambda }.
\end{equation*}

By Lemma \ref{lem6},

\begin{align*}
\left[ \mathcal{G}\left( f\right) \right] ^{\natural }\left( t,x\right) &
\leq 2\widetilde{\mathcal{M}}\left( \mathcal{G}f_{1,\lambda }\right) \left(
t,x\right) +C\left\vert f_{2,\lambda }\right\vert _{L_{\infty }\left(
E,V\right) } \\
& \leq 2\widetilde{\mathcal{M}}\left( \mathcal{G}f_{1,\lambda }\right)
\left( t,x\right) +C\delta \lambda ,\left( t,x\right) \in \mathbf{R}^{d+1},
\end{align*}

where $C$ is a constant in Lemma \ref{lem6} (independent of $\delta $ and $%
\lambda ,f)$. Fix $\delta >0$ so that $C\delta <\frac{1}{2}$. Then the above
inequality implies that

\begin{equation*}
\mathcal{G}\left( f\right) ^{\natural }\left( t,x\right) \leq 2\widetilde{%
\mathcal{M}}\left( \mathcal{G}f_{1,\lambda }\right) \left( t,x\right)
+\lambda /2,\left( t,x\right) \in \mathbf{R}^{d+1}.
\end{equation*}

Since $\left\{ \lambda \leq \mathcal{G}\left( f\right) ^{\natural }\right\}
\subseteq \left\{ \lambda \leq 4\widetilde{\mathcal{M}}\left( \mathcal{G}%
f_{1,\lambda }\right) \right\} ,$ it follows by Theorem \ref{ft} and (\ref%
{eq:L2cond}), 
\begin{align*}
& \left\vert \left[ \mathcal{G}f\right] ^{\natural }\right\vert
_{L_{p}\left( \mathbf{R}^{d+1}\right) }^{p} \\
& =p\int_{0}^{\infty }\lambda ^{p-1}\left\vert \left\{ \lambda \leq \left[ 
\mathcal{G}\left( f\right) \right] ^{\natural }\right\} \right\vert d\lambda
\leq p\int_{0}^{\infty }\lambda ^{p-1}\left\vert \left\{ \lambda \leq 4%
\widetilde{\mathcal{M}}\left( \mathcal{G}f_{1,\lambda }\right) \right\}
\right\vert d\lambda \\
& \leq C\int_{0}^{\infty }\lambda ^{p-3}\int \left\vert \widetilde{\mathcal{M%
}}\left( \mathcal{G}f_{1,\lambda }\right) \left( t,x\right) \right\vert
^{2}dtdxd\lambda \leq C\int_{0}^{\infty }\lambda ^{p-3}\int \left\vert 
\mathcal{G}f_{1,\lambda }\left( t,x\right) \right\vert ^{2}dtdxd\lambda \\
& \leq C\int_{0}^{\infty }\lambda ^{p-3}\int \left\vert f_{1,\lambda }\left(
t,x\right) \right\vert _{V}^{2}dtdxd\lambda \leq C\int_{0}^{\infty }\lambda
^{p-3}\int_{\left\{ \left\vert f\left( t,x\right) \right\vert _{V}>\delta
\lambda \right\} }\left\vert f\left( t,x\right) \right\vert
_{V}^{2}dtdxd\lambda \\
& \leq C\int \left( \int_{0}^{\left\vert f\left( t,x\right) \right\vert
_{V}/\delta }\lambda ^{p-3}d\lambda \right) \left\vert f\left( t,x\right)
\right\vert _{V}^{2}dtdx=C\int \frac{\left\vert f\left( t,x\right)
\right\vert _{V}^{p-2}}{\delta ^{p-2}\left( p-2\right) }\left\vert f\left(
t,x\right) \right\vert _{V}^{2}dtdx \\
& =C\left\vert f\right\vert _{L_{p}\left( \mathbf{R}^{d+1};V\right) }^{p}.
\end{align*}

The proof is completed by Theorem \ref{fs}.

\end{document}